\documentclass[dvipdfmx]{article}
\usepackage{a4wide}
\usepackage{amsthm}
\usepackage{graphicx}
\usepackage{amssymb}
\usepackage{amsmath}
\usepackage{ascmac}
\usepackage{setspace}
\usepackage{float}
\usepackage[dvips,usenames]{color}
\usepackage{colortbl}
\usepackage{algorithm}
\usepackage{algorithmic}
\usepackage{setspace}
\newtheorem{Theorem}[equation]{Theorem}

\newtheorem{Lemma}[equation]{Lemma}

\theoremstyle{definition}
\newtheorem{Definition}[equation]{Definition}

\theoremstyle{remark}
\newtheorem{Remark}[equation]{Remark}
\numberwithin{equation}{section}

\DeclareMathOperator{\re}{re}

\DeclareMathOperator{\ev}{ev}
\DeclareMathOperator{\id}{id}

\DeclareMathOperator{\ad}{ad}

\DeclareMathOperator{\row}{row}
\DeclareMathOperator{\col}{col}

\newcommand{\plim}[1][]{\mathop{\varprojlim}\limits_{#1}}
\newcommand{\ve}{\varepsilon}

\allowdisplaybreaks

\begin{document}
\title{Guay's affine Yangians and non-rectangular $W$-algebras}
\author{Mamoru Ueda}
\date{}
\maketitle

\begin{abstract}
We construct a non-trivial homomorphism from the Guay's affine Yangian to the universal enveloping algebra of non-rectangular $W$-algebras of type $A$. In order to construct the homomorphism, we extend the Guay's affine Yangian and its coproduct.
\end{abstract}

\section{Introduction}

A $W$-algebra appeared in the study of two dimensional conformal field theories (\cite{Z}) and has been studied by wide range mathematicians and physicists such that integrable systems, and four-dimensional gauge theories. In this paper, we relate the $W$-algebras of type $A$ to one quantum group, which is called the {\it Guay's affine Yangian}.

The Guay's affine Yangian $Y_{\hbar,\ve}(\widehat{\mathfrak{sl}}(n))$ (\cite{Gu2} and \cite{Gu1}) is a 2-parameter Yangian and is the deformation of the universal enveloping algebra of the central extension of $\mathfrak{sl}(n)[u^{\pm1},v]$. The Guay's affine Yangian $Y_{\hbar,\ve}(\widehat{\mathfrak{sl}}(n))$ has an evaluation map (\cite{Gu1} and \cite{K1})
\begin{equation*}
\ev_x\colon Y_{\hbar,\ve}(\widehat{\mathfrak{sl}}(n))\to \text{the standard degreewise completion of }U(\widehat{\mathfrak{gl}}(n))
\end{equation*}
and a coproduct (\cite{Gu1} and \cite{GNW})
\begin{equation*}
\Delta\colon Y_{\hbar,\ve}(\widehat{\mathfrak{sl}}(n))\to Y_{\hbar,\ve}(\widehat{\mathfrak{sl}}(n))\widehat{\otimes}Y_{\hbar,\ve}(\widehat{\mathfrak{sl}}(n)),
\end{equation*}
where $Y_{\hbar,\ve}(\widehat{\mathfrak{sl}}(n))\widehat{\otimes}Y_{\hbar,\ve}(\widehat{\mathfrak{sl}}(n))$ is the standard degreewise completion of $Y_{\hbar,\ve}(\widehat{\mathfrak{sl}}(n))^{\otimes2}$. 

Let us take an integer $a\geq n$. 
We extend the Guay's affine Yangian $Y_{\hbar,\ve}(\widehat{\mathfrak{sl}}(n))$ to the new associative algebra $Y^a_{\hbar,\ve}(\widehat{\mathfrak{sl}}(n))$.
One of the features of $Y^a_{\hbar,\ve}(\widehat{\mathfrak{sl}}(n))$ is that there exist the following natural algebra homomorphisms
\begin{gather*}
\Psi_1\colon Y_{\hbar,\ve}(\widehat{\mathfrak{sl}}(n))\to Y^a_{\hbar,\ve}(\widehat{\mathfrak{sl}}(n)),\\
\Psi_2\colon U(\widehat{\mathfrak{gl}}(a))\to Y^a_{\hbar,\ve}(\widehat{\mathfrak{sl}}(n))
\end{gather*}
and $Y^a_{\hbar,\ve}(\widehat{\mathfrak{sl}}(n))$ is generated by the image of $\Psi_1$ and $\Psi_2$. The algebra $Y^a_{\hbar,\ve}(\widehat{\mathfrak{sl}}(n))$ has a map corresponding to the evaluation map
\begin{gather*}
\widetilde{\ev}_x\colon Y_{\hbar,\ve}^a(\widehat{\mathfrak{sl}}(n))\to \text{the standard degreewise completion of }U(\widehat{\mathfrak{gl}}(a)).
\end{gather*}
For $a\geq b\geq n$, there also exists a map corresponding to the coproduct
\begin{gather*}
\Delta\colon Y_{\hbar,\ve}^b(\widehat{\mathfrak{sl}}(n))\to Y_{\hbar,\ve-(a-b)\hbar}^a(\widehat{\mathfrak{sl}}(n))\widehat{\otimes}Y_{\hbar,\ve}^b(\widehat{\mathfrak{sl}}(n))/\sim,
\end{gather*}
where $Y_{\hbar,\ve-(a-b)\hbar}^a(\widehat{\mathfrak{sl}}(n))\widehat{\otimes}Y_{\hbar,\ve}^b(\widehat{\mathfrak{sl}}(n))/\sim$ is the standard degreewise completion of the tensor algebra $Y_{\hbar,\ve-(a-b)\hbar}^a(\widehat{\mathfrak{sl}}(n))\otimes Y_{\hbar,\ve}^b(\widehat{\mathfrak{sl}}(n))$  divided by one relation.

A $W$-algebra $\mathcal{W}^k(\mathfrak{g},f)$ is a vertex algebra associated with a finite dimensional reductive Lie algebra $\mathfrak{g}$ and its nilpotent element $f$. It is defined by the quantized Drinfeld-Sokolov reduction (\cite{KW1} and \cite{FF}). In this paper, we consider the case that $\mathfrak{g}=\mathfrak{gl}(N)$ and its nilpotent element whose Jordan block is of type $(1^{q_1-q_2},2^{q_2-q_3},3^{q_3-q_4},\cdots,(l-1)^{q_{l-1}-q_l},l^{q_l})$, where 
\begin{gather*}
N=q_1+q_2+\cdots+q_l,\qquad q_1\geq q_2\geq\cdots\geq q_l.
\end{gather*}
In this case, there exists an injective homomorphism called {\it Miura map} (see \cite{KW1})
\begin{equation*}
\mu\colon\mathcal{W}^k(\mathfrak{g},f)\to\otimes_{i=1}^lV^{\kappa_i}(\mathfrak{gl}(q_i)),
\end{equation*}
where $V^{\kappa_i}(\mathfrak{gl}(q_i))$ is the universal affine vertex algebra associated with $\mathfrak{gl}(q_i)$ and its inner product $\kappa_i$. Taking the universal enveloping algebra of both sides in the sense of \cite{FT} and \cite{MNT}, we obtain an injective homomorphism
\begin{equation*}
\widetilde{\mu}\colon\mathcal{U}(\mathcal{W}^k(\mathfrak{g},f))\to U(\widehat{\mathfrak{gl}}(q_1))\widehat{\otimes}\cdots\widehat{\otimes} U(\widehat{\mathfrak{gl}}(q_l)),
\end{equation*}
where $U(\widehat{\mathfrak{gl}}(q_1))\widehat{\otimes}\cdots\widehat{\otimes} U(\widehat{\mathfrak{gl}}(q_l))$ is the standard degreewise completion of $\otimes_{i=1}^lU(\mathfrak{gl}(q_i))$.

In the case that $q_1=q_2=\cdots=q_l=n$, the $W$-algebra $\mathcal{W}^k(\mathfrak{g},f)$ is called the rectangular $W$-algebra of type $A$.
In this case, by a direct computation, the author \cite{U4} have constructed a surjective homomorphism
\begin{equation*}
\widehat{\Phi}\colon Y_{\hbar,\ve}^b(\widehat{\mathfrak{sl}}(n))\to\mathcal{U}(\mathcal{W}^k(\mathfrak{gl}(ln),f)),
\end{equation*}
where $\mathcal{U}(\mathcal{W}^k(\mathfrak{gl}(ln),f))$ is the universal enveloping algebra of $\mathcal{W}^k(\mathfrak{gl}(ln),f)$. In \cite{KU}, Kodera and the author showed that this homomorphism can be written down by using the coproduct (\cite{Gu1} and \cite{GNW}) and evaluation map (\cite{Gu1} and \cite{K1}) for the Guay's affine Yangian as follows;
\begin{equation*}
(\ev_0\otimes\ev_{\hbar\alpha}\otimes\cdots\otimes\ev_{\hbar(l-1)\alpha})\circ\Delta\otimes\id^{\otimes l-1})\circ\cdots\circ(\Delta\otimes\id)\circ\Delta=\widetilde{\mu}\circ\widehat{\Phi},
\end{equation*}
where $\alpha=k+(l-1)n$.

We extend this result  to the general nilpotent element.
In finite setting, Brundan-Kleshchev \cite{BK} gave a surjective homomorphism from a shifted Yangian, which is a subalgebra of the finite Yangian associated with $\mathfrak{gl}(n)$, to a finite $W$-algebra (\cite{Pr}) of type $A$ for its general nilpotent element. A finite $W$-algebra $\mathcal{W}^{\text{fin}}(\mathfrak{g},f)$ is an associative algebra associated with a reductive Lie algebra $\mathfrak{g}$ and its nilpotent element $f$ and is a finite analogue of a $W$-algebra $\mathcal{W}^k(\mathfrak{g},f)$ (\cite{DSK1} and \cite{A1}). In \cite{DKV}, De Sole, Kac and Valeri constructed a homomorphism from the finite Yangian of type $A$ to the finite $W$-algebras of type $A$ by using the Lax operator, which is a restriction of the homomorphism given by Brundan-Kleshchev in \cite{BK}.

Motivated by the work of De Sole, Kac and Valeri \cite{DKV}, by a direct computation, the author \cite{U5} constructed a homomorphism from the Guay's affine Yangian $Y_{\hbar,\ve}(\widehat{\mathfrak{sl}}(n))$ to the universal enveloping algebra of $\mathcal{W}^k(\mathfrak{gl}(m+n),(1^{m-n},2^n))$.
In this article, we constructed a homomorphism from the Guay's affine Yangian $Y_{\hbar,\ve}(\widehat{\mathfrak{sl}}(q_l))$ to the universal enveloping algebra of the $W$-algebra $\mathcal{W}^k(\mathfrak{gl}(N),f)$.
\begin{Theorem}
Let $q_l\geq3$. We assume that $\dfrac{\ve}{\hbar}=-(k+N)$.
Then, there exists an algebra homomorphism 
\begin{equation*}
\Phi\colon Y_{\hbar,\ve}(\widehat{\mathfrak{sl}}(q_l))\to \mathcal{U}(\mathcal{W}^{k}(\mathfrak{gl}(N),f))
\end{equation*} 
satisfying
\begin{equation*}
\bigotimes_{i=1}^l\limits\widetilde{\ev}_{a_i}\circ\Delta^{q_1,q_2}\otimes\id^{\otimes l-1})\circ\cdots\circ(\Delta^{q_{l-2},q_{l-1}}\otimes\id)\circ\Delta^{q_{l-1},q_l}\circ\Psi_1=\widetilde{\mu}\circ\Phi,
\end{equation*}
where $a_i=-\hbar\sum_{y=i+1}^l\limits (k+N-q_i)$ and $\ve_i=\ve-(q_i-q_l)\hbar$.
\end{Theorem}

We hope that this theorem will help to resolve the genralized AGT (Alday-Gaiotto-Tachikawa) conjecture.
The AGT conjecture suggests that there exists a representation of the principal $W$-algebra of type $A$ on the equivariant homology space of the moduli space of $U(r)$-instantons. Schiffmann and Vasserot \cite{SV} gave this representation by using an action of the Yangian associated with $\widehat{\mathfrak{gl}}(1)$ on this equivariant homology space. 
It is conjectured in \cite{CE} that an action of an iterated $W$-algebra of type $A$ on the equivariant homology space of the affine Laumon space will be given through an action of an affine shifted Yangian constructed in \cite {FT}.

\section{Guay's affine Yangian}

Let us recall the definition of the Guay's affine Yangian. The Guay's affine Yangian $Y_{\ve_1,\ve_2}(\widehat{\mathfrak{sl}}(n))$ was first introduced by Guay (\cite{Gu2} and \cite{Gu1}) and is the deformation of the universal enveloping algebra of the central extension of $\mathfrak{sl}(n)[u^{\pm1},v]$.
\begin{Definition}\label{Prop32}
Let $n\geq3$ and an $n\times n$ matrix $(a_{i,j})_{1\leq i,j\leq n}$ be
\begin{gather*}
a_{ij} =
	\begin{cases}
	2 &\text{if } i=j, \\
	         -1&\text{if }j=i\pm1,\\
	        1 &\text{if }(i,j)=(0,n-1),(n-1,0),\\
		0  &\text{otherwise.}
	\end{cases}
\end{gather*}
The Guay's affine Yangian $Y_{\hbar,\ve}(\widehat{\mathfrak{sl}}(n))$ is the associative algebra generated by $X_{i,r}^{+}, X_{i,r}^{-}, H_{i,r}$ $(i \in \{0,1,\cdots, n-1\}, r = 0,1)$ subject to the following defining relations:
\begin{gather}
[H_{i,r}, H_{j,s}] = 0,\label{Eq2.1}\\
[X_{i,0}^{+}, X_{j,0}^{-}] = \delta_{ij} H_{i, 0},\label{Eq2.2}\\
[X_{i,1}^{+}, X_{j,0}^{-}] = \delta_{ij} H_{i, 1} = [X_{i,0}^{+}, X_{j,1}^{-}],\label{Eq2.3}\\
[H_{i,0}, X_{j,r}^{\pm}] = \pm a_{ij} X_{j,r}^{\pm},\label{Eq2.4}\\
[\tilde{H}_{i,1}, X_{j,0}^{\pm}] = \pm a_{ij}\left(X_{j,1}^{\pm}\right)\text{ if }(i,j)\neq(0,n-1),(n-1,0),\label{Eq2.5}\\
[\tilde{H}_{0,1}, X_{n-1,0}^{\pm}] = \mp \left(X_{n-1,1}^{\pm}-(\ve+\dfrac{n}{2}\hbar) X_{n-1, 0}^{\pm}\right),\label{Eq2.6}\\
[\tilde{H}_{n-1,1}, X_{0,0}^{\pm}] = \mp \left(X_{0,1}^{\pm}+(\ve+\dfrac{n}{2}\hbar) X_{0, 0}^{\pm}\right),\label{Eq2.7}\\
[X_{i, 1}^{\pm}, X_{j, 0}^{\pm}] - [X_{i, 0}^{\pm}, X_{j, 1}^{\pm}] = \pm a_{ij}\dfrac{\hbar}{2} \{X_{i, 0}^{\pm}, X_{j, 0}^{\pm}\}\text{ if }(i,j)\neq(0,n-1),(n-1,0),\label{Eq2.8}\\
[X_{0, 1}^{\pm}, X_{n-1, 0}^{\pm}] - [X_{0, 0}^{\pm}, X_{n-1, 1}^{\pm}]= \mp\dfrac{\hbar}{2} \{X_{0, 0}^{\pm}, X_{n-1, 0}^{\pm}\} - (\ve+\dfrac{n}{2}\hbar) [X_{0, 0}^{\pm}, X_{n-1, 0}^{\pm}],\label{Eq2.9}\\
(\ad X_{i,0}^{\pm})^{1-a_{ij}} (X_{j,0}^{\pm})= 0 \text{ if }i \neq j, \label{Eq2.10}
\end{gather}
where $\tilde{H}_{i,1}=H_{i,1}-\dfrac{\hbar}{2}H_{i,0}^2$.
\end{Definition}
\begin{Remark}
The defining relations of $Y_{\hbar,\ve}(\widehat{\mathfrak{sl}}(n))$ are different from those of $Y_{\ve_1,\ve_2}(\widehat{\mathfrak{sl}}(n))$ which is called the Guay's affine Yangian in \cite{K1}. In \cite{K1}, generators of $Y_{\ve_1,\ve_2}(\widehat{\mathfrak{sl}}(n))$ are denoted by 
\begin{equation*}
\{x^\pm_{i,r},h_{i,r}\mid0\leq i\leq n-1,r\in\mathbb{Z}_{\geq0}\}
\end{equation*}
with 2-parameters $\ve_1$ and $\ve_2$.
Actually, the algebra $Y_{\hbar,\ve}(\widehat{\mathfrak{sl}}(n))$ is isomorphic to  $Y_{\ve_1,\ve_2}(\widehat{\mathfrak{sl}}(n))$.
The isomorphism $\Psi$ from $Y_{\hbar,\ve}(\widehat{\mathfrak{sl}}(n))$ to $Y_{\ve_1,\ve_2}(\widehat{\mathfrak{sl}}(n))$ is given by
\begin{gather*}
\Psi(H_{i,0})=h_{i,0},\quad\Psi(X^\pm_{i,0})=x^\pm_{i,0},\\
\Psi(H_{i,1})=\begin{cases}
h_{0,1}&\text{ if $i=0$},\\
h_{i,1}+\dfrac{i}{2}(\ve_1-\ve_2)h_{i,0}&\text{ if $i\neq0$},
\end{cases}\\
\hbar=\ve_1+\ve_2,\qquad
\ve=-n\ve_2.
\end{gather*}
In this paper, we do not use $Y_{a,b}(\widehat{\mathfrak{sl}}(n))$ in the meaning of \cite{K1}.
\end{Remark}

Let us recall the evaluation map for the Guay's affine Yangian (see \cite{Gu1} and \cite {K1}). We set a Lie algebra 
\begin{equation*}
\widehat{\mathfrak{gl}}(n)^c=\Big(\bigoplus_{s\in\mathbb{Z}}\limits\bigoplus_{1\leq i,j\leq n}\limits E_{i,j}t^s\Big)\oplus\mathbb{C}c\oplus\mathbb{C}z
\end{equation*}
whose commutator relations are determined by
\begin{gather*}
[E_{p,q}t^s,E_{i,j}t^u]=\delta_{i,q}E_{p,j}t^{s+u}-\delta_{p,j}E_{i,q}t^{s+u}+s\delta_{i,q}\delta_{p,j}\delta_{s+u,0}c+s\delta_{p,q}\delta_{i,j}\delta_{s+u,0}z,\\
\text{$c$ and $z$ are central elements}.
\end{gather*}
Let us set $U^1(\widehat{\mathfrak{gl}}(n)^c)$ as $U(\widehat{\mathfrak{gl}}(n)^c)/U(\widehat{\mathfrak{gl}}(n)^c)(z-1)$.
Next, we introduce a completion of $U^1(\widehat{\mathfrak{gl}}(n)^c)$ following \cite{MNT}. We set the grading of $U^1(\widehat{\mathfrak{gl}}(n))$ as $\text{deg}(E_{i,j}t^s)=s$ and $\text{deg}(c)=0$.
Then, $U^1(\widehat{\mathfrak{gl}}(n)^c)$ becomes a graded algebra and we denote the set of the degree $d$ elements of $U^1(\widehat{\mathfrak{gl}}(n)^c)$ by $U^1(\widehat{\mathfrak{gl}}(n)^c)_d$. 
We obtain the completion
\begin{equation*}
U^1(\widehat{\mathfrak{gl}}(n)^c)_{{\rm comp}}=\bigoplus_{d\in\mathbb{Z}}U^1(\widehat{\mathfrak{gl}}(n)^c)_{{\rm comp},d},
\end{equation*}
where
\begin{equation*}
U^1(\widehat{\mathfrak{gl}}(n))^c_{{\rm comp},d}=\plim[N]U^1(\widehat{\mathfrak{gl}}(n)^c)_{d}/\sum_{r>N}\limits U^1(\widehat{\mathfrak{gl}}(n)^c)_{d-r}U^1(\widehat{\mathfrak{gl}}(n)^c)_{r}.
\end{equation*}
The evaluation map for the Guay's affine Yangian is a non-trivial homomorphism from the Guay's affine Yangian to $U^1(\widehat{\mathfrak{gl}}(n)^c)_{{\rm comp}}$. 
Here after, we denote
\begin{gather*}
h_i=\begin{cases}
E_{n,n}-E_{1,1}+c&(i=0),\\
E_{ii}-E_{i+1,i+1}&(1\leq i\leq n-1),
\end{cases}\\
x^+_i=\begin{cases}
E_{n,1}t&(i=0),\\
E_{i,i+1}&(1\leq i\leq n-1),
\end{cases}
\quad x^-_i=\begin{cases}
E_{1,n}t^{-1}&(i=0),\\
E_{i+1,i}&(1\leq i\leq n-1).
\end{cases}
\end{gather*}
\begin{Theorem}[Section 6 in \cite{Gu1} and Theorem~3.8 in \cite{K1}]\label{thm:main}
Set $c=\dfrac{-n\hbar-\ve}{\hbar}$.
Then, there exists an algebra homomorphism 
\begin{equation*}
\ev_{\hbar,\ve} \colon Y_{\hbar,\ve}(\widehat{\mathfrak{sl}}(n)) \to U^1(\widehat{\mathfrak{gl}}(n)^c)_{{\rm comp}}
\end{equation*}
uniquely determined by 
\begin{gather*}
	\ev_{\hbar,\ve}(X_{i,0}^{+}) = x_{i}^{+}, \quad \ev_{\hbar,\ve}(X_{i,0}^{-}) = x_{i}^{-},\quad \ev_{\hbar,\ve}(H_{i,0}) = h_{i},
\end{gather*}
\begin{gather*}
	\ev_{\hbar,\ve}(H_{i,1}) = \begin{cases}
		\hbar ch_{0} -\hbar E_{n,n} (E_{1,1}-c) \\
		\ +\hbar \displaystyle\sum_{s \geq 0} \limits\displaystyle\sum_{k=1}^{n}\limits E_{n,k}t^{-s} E_{k,n}t^s-\hbar \displaystyle\sum_{s \geq 0} \displaystyle\sum_{k=1}^{n}\limits E_{1,k}t^{-s-1} E_{k,1}t^{s+1}\\ \qquad\qquad\qquad\qquad\qquad\qquad\qquad\qquad\qquad\qquad\qquad\qquad\qquad\qquad\text{ if $i = 0$},\\
		-\dfrac{i}{2}\hbar h_{i} -\hbar E_{i,i}E_{i+1,i+1} \\
		\quad+ \hbar\displaystyle\sum_{s \geq 0}  \limits\displaystyle\sum_{k=1}^{i}\limits E_{i,k}t^{-s} E_{k,i}t^s+\hbar\displaystyle\sum_{s \geq 0} \limits\displaystyle\sum_{k=i+1}^{n}\limits E_{i,k}t^{-s-1} E_{k,i}t^{s+1}\\
\quad-\hbar\displaystyle\sum_{s \geq 0}\limits\displaystyle\sum_{k=1}^{i}\limits E_{i+1,k}t^{-s}E_{k,i+1}t^s-\hbar\displaystyle\sum_{s \geq 0}\limits\displaystyle\sum_{k=i+1}^{n} \limits E_{i+1,k}t^{-s-1}E_{k,i+1}t^{s+1}\\
\qquad\qquad\qquad\qquad\qquad\qquad\qquad\qquad\qquad\qquad\qquad\qquad\qquad\qquad \text{ if $i \neq 0$},
	\end{cases}
\end{gather*}
\begin{align*}
\ev_{\hbar,\ve}(X^+_{i,1})&=\begin{cases}
\hbar cx_{0}^{+} + \hbar \displaystyle\sum_{s \geq 0} \limits\displaystyle\sum_{k=1}^{n}\limits E_{n,k}t^{-s} E_{k,1}t^{s+1}\\
\qquad\qquad\qquad\qquad\qquad\qquad\qquad\qquad\qquad\qquad\qquad\qquad\qquad\qquad \text{ if $i = 0$},\\
-\dfrac{i}{2}\hbar x_{i}^{+}+ \hbar \displaystyle\sum_{s \geq 0}\limits\displaystyle\sum_{k=1}^i\limits E_{i,k}t^{-s} E_{k,i+1}t^s+\hbar \displaystyle\sum_{s \geq 0}\limits\displaystyle\sum_{k=i+1}^{n}\limits E_{i,k}t^{-s-1} E_{k,i+1}t^{s+1}\\
\qquad\qquad\qquad\qquad\qquad\qquad\qquad\qquad\qquad\qquad\qquad\qquad\qquad\qquad \text{ if $i \neq 0$},
\end{cases}
\end{align*}
\begin{align*}
\ev_{\hbar,\ve}(X^-_{i,1})&=\begin{cases}
\hbar cx_{0}^{-} +\hbar \displaystyle\sum_{s \geq 0} \limits\displaystyle\sum_{k=1}^{n}\limits E_{1,k}t^{-s-1} E_{k,n}t^{s},\\
\qquad\qquad\qquad\qquad\qquad\qquad\qquad\qquad\qquad\qquad\qquad\qquad\qquad\qquad \text{ if $i = 0$},\\
-\dfrac{i}{2}\hbar x_{i}^{-}+\hbar \displaystyle\sum_{s \geq 0}\limits\displaystyle\sum_{k=1}^i\limits E_{i+1,k}t^{-s} E_{k,i}t^{s}+ \hbar \displaystyle\sum_{s \geq 0}\limits\displaystyle\sum_{k=i+1}^{n}\limits E_{i+1,k}t^{-s-1} E_{k,i}t^{s+1}\\
\qquad\qquad\qquad\qquad\qquad\qquad\qquad\qquad\qquad\qquad\qquad\qquad\qquad\qquad \text{ if $i \neq 0$}.
\end{cases}
\end{align*}
\end{Theorem}
We recall the coproduct for the Guay's affine Yangian.
Let $\Delta^+$ (resp. $\Delta^+_{\re}$) be the set of positive roots (resp. positive real roots) of $\widehat{\mathfrak{sl}}(n)$. We denote the multiplicity of a root $\gamma$ by $p(\gamma)$. We take root vectors $\{x_{\pm\gamma}^{(r)} \mid 1\leq r\leq p(\gamma)\}$ for $\gamma\in\Delta^+$ satisfying $(x^{(r)}_{\gamma},x^{(s)}_{-\gamma})=\delta_{r,s}$, where $(\ ,\ )$ is the standard invariant symmetric bilinear form. We denote the simple roots of $\widehat{\mathfrak{sl}}(n)$ by $\{\alpha_i\mid 0\leq i\leq n-1\}$.

By Theorem 6.1 in \cite{Gu1} and Theorem 6.9 in \cite{GRW}, we have an embedding $\xi$ from $U(\widehat{\mathfrak{sl}}(n))\subset U(\widehat{\mathfrak{gl}}(n))^c$ to $Y_{\hbar,\ve}(\widehat{\mathfrak{sl}}(n))$ determined by
\begin{gather*}
\xi(h_i)=H_{i,0},\qquad\xi(x^\pm_i)=X^\pm_{i,0}.
\end{gather*}
We identify $U(\widehat{\mathfrak{sl}}(n))$ and its image via $\xi$. 

We set the degree of the Guay's affine Yangian as follows;
\begin{equation*}
\text{deg}(H_{i,r})=0,\quad\text{deg}(X^\pm_{i,r})=\pm\delta_{i,0}.
\end{equation*}
By this degree, we can define the standard degree completion of $Y_{\hbar,\ve}(\widehat{\mathfrak{sl}}(n))^{\otimes 2}$. We denote it by $Y_{\hbar,\ve}(\widehat{\mathfrak{sl}}(n))\widehat{\otimes} Y_{\hbar,\ve}(\widehat{\mathfrak{sl}}(n))$.
\begin{Theorem}[Theorem~5.2 in \cite{GNW}]
There exists an algebra homomorphism
\begin{equation*}
\Delta\colon Y_{\hbar,\ve}(\widehat{\mathfrak{sl}}(n))\to Y_{\hbar,\ve}(\widehat{\mathfrak{sl}}(n))\widehat{\otimes} Y_{\hbar,\ve}(\widehat{\mathfrak{sl}}(n))
\end{equation*}
determined by
\begin{gather*}
\Delta(H_{i,0})=H_{i,0}\otimes1+1\otimes H_{i,0},\\
\Delta(X^\pm_{i,0})=X^\pm_{i,0}\otimes1+1\otimes X^\pm_{i,0},\\
\begin{align*}
\Delta(H_{i,1})&=H_{i,1}\otimes1+1\otimes H_{i,1}\\
&\quad+\hbar(H_{i,0}\otimes H_{i,0}-\displaystyle\sum_{\gamma\in\Delta^+_{\re}}\limits (\alpha_i,\gamma)x^{(1)}_{-\gamma}\otimes x^{(1)}_\gamma),\\
\Delta(X^+_{i,1})&=X^+_{i,1}\otimes1+1\otimes X^+_{i,1}\\
&\quad+\hbar(H_{i,0}\otimes X^+_{i,0}-\displaystyle\sum_{\gamma\in\Delta^+}\limits\sum_{r=1}^{p(\gamma)}\limits x^{(r)}_{-\gamma}\otimes[x^+_i,x^{(r)}_\gamma]),\\
\Delta(X^-_{0,1})&=X^-_{i,1}\otimes1+1\otimes X^-_{i,1}\\
&\quad+\hbar(H_{i,0}\otimes X^+_{i,0}+\displaystyle\sum_{\gamma\in\Delta^+}\limits\sum_{r=1}^{p(\gamma)}\limits [x^-_i,x^{(r)}_{-\gamma}]\otimes x^{(r)}_\gamma).
\end{align*}
\end{gather*}
\end{Theorem}

\section{Extension to the new Yangian}
We extend the definition of the Guay's affine Yangian.
\begin{Definition}
Let $a\geq n$. We define 
$\widetilde{Y}^{a}_{\hbar,\ve}(\widehat{\mathfrak{sl}}(n))$ by the associative algebra whose generators are $\{H_{i,r},X^\pm_{i,r}\mid0\leq i\leq n-1,r\in\mathbb{Z}_{\geq0}\}$ and $\{e_{i,j}t^s,c_a\mid 1\leq i,j\leq n,s\in\mathbb{Z}\}$ with the following defining relations;
\begin{gather*}
\text{the defining relations \eqref{Eq2.1}-\eqref{Eq2.10}},\\
[e_{i,j}t^s,e_{u,v}t^w]=\delta_{j,u}e_{i,v}t^{s+w}-\delta_{i,v}e_{u,j}t^{s+w}+s\delta_{s+w,0}\delta_{i,v}\delta_{u,j}c_a+\delta_{i,j}\delta_{u,v},\\
\text{$c_a$ is a central element},\\
H_{i,0}=\begin{cases}
e_{n,n}-e_{1,1}+c_a\text{ if }i=0,\\
e_{i,i}-e_{i+1,i+1}\text{ if }i\neq 0,
\end{cases}\\
X^+_{i,0}=\begin{cases}
e_{n,1}t\text{ if }i=0,\\
e_{i,i+1}\text{ if }i\neq 0,
\end{cases}X^-_{i,0}=\begin{cases}
e_{1,n}t^{-1}\text{ if }i=0,\\
e_{i+1,i}\text{ if }i\neq 0.
\end{cases}
\end{gather*}
\end{Definition}
We set the degree on $\widetilde{Y}^{a}_{\hbar,\ve}(\widehat{\mathfrak{sl}}(n))$ as
\begin{equation*}
\text{deg}(H_{i,r})=0,\text{deg}(X^\pm_{i,r})=\pm \delta_{i,0},\text{deg}(e_{i,j}t^s)=s,\text{deg}(c_a)=0.
\end{equation*}
We define $\widehat{Y}^{a}_{\hbar,\ve}(\widehat{\mathfrak{sl}}(n))$ as the standard degreewise completion of $\widetilde{Y}^{a}_{\hbar,\ve}(\widehat{\mathfrak{sl}}(n))$. Let us define $\widehat{\ev}_{\hbar,\ve}(H_{i,1})$ as an element of $\widehat{Y}^{a}_{\hbar,\ve}(\widehat{\mathfrak{sl}}(n))$ in the same formula as the one in \eqref{thm:main}. By a direct computation, we obtain
\begin{align}
&\quad[\widehat{\ev}_{\hbar,\ve}(H_{i,1}),e_{v,j}t^w]\nonumber\\
&=\dfrac{i}{2}\hbar\delta_{i,j}e_{j,v}t^w-\dfrac{i}{2}\hbar\delta_{i+1,j}e_{j,v}t^w+\hbar \delta_{i,j}e_{v,j}t^we_{i+1,i+1}+\hbar\delta_{i+1,j} e_{i,i}e_{v,j}t^w\nonumber\\
&\quad-\hbar\displaystyle\sum_{s \geq 0} \limits\delta(j\leq i)e_{i,j}t^{-s}e_{v,i}t^{s+w}-\hbar\displaystyle\sum_{s \geq 0} \limits\sum_{u=1}^i\limits\delta_{i,j}e_{v,u}t^{w-s}e_{u,i}t^s\nonumber\\
&\quad-\hbar\displaystyle\sum_{s \geq 0} \limits\delta(j>i)e_{i,j}t^{-s-1}e_{v,i}t^{s+w+1}-\hbar\displaystyle\sum_{s \geq 0} \limits\displaystyle\sum_{u=i+1}^{n}\limits \delta_{i,j}e_{v,u}t^{w-s-1}e_{u,i}t^{s+1}\nonumber\\
&\quad+\hbar\displaystyle\sum_{s \geq 0}\limits \delta(j\leq i)e_{i+1,j}t^{-s}e_{v,i+1}t^{s+w}+\hbar\displaystyle\sum_{s \geq 0}\limits\displaystyle\sum_{u=1}^{i}\limits \delta_{i+1,j}e_{v,u}t^{w-s}e_{u,i+1}t^s\nonumber\\
&\quad+\hbar\displaystyle\sum_{s \geq 0}\limits\delta(j>i)e_{i+1,j}t^{-s-1}e_{v,i+1}t^{s+w+1}+\hbar\displaystyle\sum_{s \geq 0}\limits\displaystyle\sum_{u=i+1}^{n} \limits \delta_{i+1,j}e_{v,u}t^{w-s-1}e_{u,i+1}t^{s+1},\label{Eq2.111}\\
&\quad[\widehat{\ev}_{\hbar,\ve}(H_{i,1}),e_{j,v}t^w]\nonumber\\
&=-\dfrac{i}{2}\hbar\delta_{i,j}e_{j,v}t^w+\dfrac{i}{2}\hbar\delta_{i+1,j}e_{j,v}t^w-\hbar\delta_{i,j}e_{j,v}t^we_{i+1,i+1}-\hbar\delta_{i+1,j}e_{i,i}e_{j,v}t^w\nonumber\\
&\quad+\hbar\displaystyle\sum_{s \geq 0}  \limits\displaystyle\sum_{u=1}^{i}\limits\delta_{i,j}e_{i,u}t^{-s}e_{u,v}t^{s+w}+\hbar\displaystyle\sum_{s \geq 0}  \limits\delta(j\leq i)e_{i,v}t^{w-s}e_{j,i}t^s\nonumber\\
&\quad+\hbar\displaystyle\sum_{s \geq 0} \limits\displaystyle\sum_{u=i+1}^{n}\limits\delta_{i,j}e_{i,u}t^{-s-1}e_{u,v}t^{s+w+1}+\hbar\displaystyle\sum_{s \geq 0} \limits\delta(j>i) e_{i,v}t^{w-s-1}e_{j,i}t^{s+1}\nonumber\\
&\quad-\hbar\displaystyle\sum_{s \geq 0}\limits\displaystyle\sum_{u=1}^{i}\limits \delta_{i+1,j}e_{i+1,u}t^{-s}e_{u,v}t^{s+w}-\hbar\displaystyle\sum_{s \geq 0}\limits\delta(j\leq i)e_{i+1,v}t^{w-s}e_{j,i+1}t^s\nonumber\\
&\quad-\hbar\displaystyle\sum_{s \geq 0}\limits\displaystyle\sum_{u=i+1}^{n} \limits\delta_{i+1,j}e_{i+1,u}t^{-s-1}e_{u,v}t^{s+w+1}\nonumber\\
&\quad-\hbar\displaystyle\sum_{s \geq 0}\limits\delta(j>i)e_{i+1,v}t^{w-s-1}e_{j,i+1}t^{s+1}\label{Eq2.113}
\end{align}
for all $i\neq0$, $1\leq j\leq n$ and $n<v\leq b$. By a direct computation, we also obtain
\begin{align}
&\quad[\widehat{\ev}_{\hbar,\ve}(H_{0,1}),e_{v,j}t^w]\nonumber\\
&=-\hbar c_a\delta_{n,j}e_{v,j}t^w+\hbar c_a\delta_{1,j}e_{v,j}t^w +\hbar \delta_{1,j}e_{n,n}e_{v,j}t^w+\delta_{n,j}\hbar e_{v,j}t^w e_{1,1}-\delta_{n,j}\hbar c_ae_{v,j}t^w\nonumber\\
&\quad-\hbar \displaystyle\sum_{s \geq 0} \limits e_{n,j}t^{-s} e_{v,n}t^{s+w}-\hbar \displaystyle\sum_{s \geq 0} \limits\displaystyle\sum_{u=1}^{n}\limits \delta_{n,j}e_{v,u}t^{w-s} e_{u,n}t^s\nonumber\\
&\quad+\hbar \displaystyle\sum_{s \geq 0}  e_{1,j}t^{-s-1} e_{v,1}t^{w+s+1}+\hbar \displaystyle\sum_{s \geq 0} \displaystyle\sum_{u=1}^{n}\limits\delta_{j,1} e_{v,u}t^{w-s-1} e_{u,1}t^{s+1},\label{Eq2.115}\\
&\quad[\widehat{\ev}_{\hbar,\ve}(H_{0,1}),e_{j,v}t^w]\nonumber\\
&=\hbar c_a\delta_{j,n}e_{j,v}t^w-\hbar c_a\delta_{1,j}e_{j,v}t^w-\hbar\delta_{1,j} e_{n,n}e_{j,v}t^w-\hbar \delta_{j,n}e_{j,v}t^we_{1,1}+\hbar \delta_{j,n}c_ae_{j,v}t^w\nonumber\\
&\quad+\hbar \displaystyle\sum_{s \geq 0} \limits\displaystyle\sum_{u=1}^{n}\limits \delta_{j,n}e_{n,u}t^{-s} e_{u,v}t^{s+w}+\hbar \displaystyle\sum_{s \geq 0} \limits e_{n,v}t^{w-s} e_{j,n}t^s\nonumber\\
&\quad-\hbar \displaystyle\sum_{s \geq 0} \displaystyle\sum_{u=1}^{n}\limits \delta_{1,j}e_{1,u}t^{-s-1} e_{u,v}t^{w+s+1}-\hbar \displaystyle\sum_{s \geq 0}  e_{1,v}t^{w-s-1} e_{j,1}t^{s+1}\label{Eq2.118}
\end{align}
for all $1\leq j\leq n$ and $n<v\leq b$. 

We set an associative algebra $Y^{a}_{\hbar,\ve}(\widehat{\mathfrak{sl}}(n))$ as a quotient algebra divided by
\begin{gather}
[H_{i,1},e_{v,j}t^w]=[\widehat{\ev}_{\hbar,\ve}(H_{i,1}),e_{v,j}t^w],\label{Eq2.11}\\
[H_{i,1},e_{j,v}t^w]=[\widehat{\ev}_{\hbar,\ve}(H_{i,1}),e_{j,v}t^w],\label{Eq2.12}\\
[H_{i-1,1},e_{v,i}t^w]+[H_{i,1},e_{v,i}t^w]=[\widehat{\ev}_{\hbar,\ve}(H_{i-1,1}),e_{v,i}t^w]+[\widehat{\ev}_{\hbar,\ve}(H_{i,1}),e_{v,i}t^w],\label{Eq2.13}\\
[H_{i-1,1},e_{v,i}t^w]+[H_{i,1},e_{v,i}t^w]=[\widehat{\ev}_{\hbar,\ve}(H_{i-1,1}),e_{v,i}t^w]+[\widehat{\ev}_{\hbar,\ve}(H_{i,1}),e_{v,i}t^w],\label{Eq2.14}\\
[H_{0,1},e_{v,j}t^w]=[\widehat{\ev}_{\hbar,\ve}(H_{0,1}),e_{v,j}t^w],\label{Eq2.15}\\
[H_{0,1},e_{j,v}t^w]=[\widehat{\ev}_{\hbar,\ve}(H_{0,1}),e_{j,v}t^w],\label{Eq2.16}\\
[H_{0,1},e_{v,n}t^w]+[H_{n-1,1},e_{v,n}t^w]=[\widehat{\ev}_{\hbar,\ve}(H_{0,1}),e_{v,n}t^w]+[\widehat{\ev}_{\hbar,\ve}(H_{n-1,1}),e_{v,n}t^w],\label{Eq2.17}\\
[H_{0,1},e_{v,1}t^w]+[H_{1,1},e_{v,1}t^w]=[\widehat{\ev}_{\hbar,\ve}(H_{1,1}),e_{v,1}t^w]+[\widehat{\ev}_{\hbar,\ve}(H_{1,1}),e_{v,1}t^w],\label{Eq2.18}\\
[H_{0,1},e_{n,v}t^w]+[H_{n-1,1},e_{n,v}t^w]=[\widehat{\ev}_{\hbar,\ve}(H_{0,1}),e_{n,v}t^w]+[\widehat{\ev}_{\hbar,\ve}(H_{n-1,1}),e_{n,v}t^w],\label{Eq2.19}\\
[H_{0,1},e_{1,v}t^w]+[H_{1,1},e_{1,v}t^w]=[\widehat{\ev}_{\hbar,\ve}(H_{1,1}),e_{1,v}t^w]+[\widehat{\ev}_{\hbar,\ve}(H_{1,1}),e_{1,v}t^w].\label{Eq2.20}
\end{gather}
By the definition of $Y^{a}_{\hbar,\ve}(\widehat{\mathfrak{sl}}(n))$, we have two homomorphisms;
\begin{equation*}
\Psi_1\colon Y_{\hbar,\ve}(\widehat{\mathfrak{sl}}(n))\to Y^{a}_{\hbar,\ve}(\widehat{\mathfrak{sl}}(n))
\end{equation*}
determined by
\begin{equation*}
\Psi_1(H_{i,r})=H_{i,r},\qquad\Psi_1(X^\pm_{i,r})=X^\pm_{i,r}
\end{equation*}
and
\begin{equation*}
\Psi_2\colon U(\widehat{\mathfrak{gl}}(n)^{c_a})\to Y^{a}_{\hbar,\ve}(\widehat{\mathfrak{sl}}(n))
\end{equation*}
determined by
\begin{equation*}
\Psi_2(e_{i,j}t^s)=e_{i,j}t^s\qquad\Psi_2(c_a)=c_a.
\end{equation*}

By the definition of $Y^{a}_{\hbar,\ve}(\widehat{\mathfrak{sl}}(n))$, we find that we can construct a non-trivial homomorphism from $Y^{a}_{\hbar,\ve}(\widehat{\mathfrak{sl}}(n))$ to the standard degreewise completion of the universal enveloping algebra of $\widehat{\mathfrak{gl}}(a)$ as follows.
\begin{Theorem}\label{evaluation}
For $x\in\mathbb{C}$, there exists an algebra homomorphism
\begin{equation*}
\widetilde{\ev}_{\hbar,\ve}^x\colon Y^{a}_{\hbar,\ve}(\widehat{\mathfrak{sl}}(n))\to U(\widehat{\mathfrak{gl}}(n)^{c_a})_{\text{comp}}
\end{equation*}
determined by
\begin{gather*}
\widetilde{\ev}_{\hbar,\ve}^x(e_{i,j}t^s)=e_{i,j}t^s,\qquad
\widetilde{\ev}_{\hbar,\ve}^x(c_a)=-\dfrac{n\hbar+\ve}{\hbar}\\
\widetilde{\ev}_{\hbar,\ve}^x(H_{i,1})=\ev_{\hbar,\ve}(H_{i,1})+x H_{i,0},\\
\widetilde{\ev}_{\hbar,\ve}^x(X^\pm_{i,1})=\ev_{\hbar,\ve}(X^\pm_{i,1})+x X^\pm_{i,0}.
\end{gather*}
\end{Theorem}
We can construct a map corresponding to a coproduct. Let $a\geq b\geq n$. 
We take a degree for $Y^{a}_{\hbar,\ve}(\widehat{\mathfrak{sl}}(n))\otimes Y^{b}_{\hbar,\ve}(\widehat{\mathfrak{sl}}(n))$ determined by
\begin{gather*}
\text{deg}(H_{i,r}\otimes 1)=\text{deg}(1\otimes H_{i,r})=0,\\
\text{deg}(X^\pm_{i,r}\otimes1)=\text{deg}(1\otimes X^\pm_{i,r})=\pm \delta_{i,0},\\
\text{deg}(e_{i,j}t^s\otimes1)=\text{deg}(1\otimes e_{i,j}t^s)=s,\\
\text{deg}(c_a\otimes 1)=\text{deg}(1\otimes c_b)=0.
\end{gather*}
We set $Y^{a}_{\hbar,\ve}(\widehat{\mathfrak{sl}}(n))\widehat{\otimes} Y^{b}_{\hbar,\ve}(\widehat{\mathfrak{sl}}(n))$ as the standard degreewise completion of $Y^{a}_{\hbar,\ve}(\widehat{\mathfrak{sl}}(n))\otimes Y^{b}_{\hbar,\ve}(\widehat{\mathfrak{sl}}(n))$. Moreover, we set $Y^{a}_{\hbar,\ve}(\widehat{\mathfrak{sl}}(n))\widetilde{\otimes} Y^{b}_{\hbar,\ve}(\widehat{\mathfrak{sl}}(n))$ as a quotient algebra of $Y^{a}_{\hbar,\ve}(\widehat{\mathfrak{sl}}(n))\widehat{\otimes} Y^{b}_{\hbar,\ve}(\widehat{\mathfrak{sl}}(n))$ divided by 
\begin{equation*}
c_a\otimes1-1\otimes c_b=-(a-b).
\end{equation*}
\begin{Theorem}\label{Coproduct}
There exists an algebra homomorphism
\begin{equation*}
\Delta^{a,b}\colon Y^{b}_{\hbar,\ve-(a-b)\hbar}(\widehat{\mathfrak{sl}}(n))\to Y^{a}_{\hbar,\ve}(\widehat{\mathfrak{sl}}(n))\widetilde{\otimes} Y^{b}_{\hbar,\ve}(\widehat{\mathfrak{sl}}(n))
\end{equation*}
determined by
\begin{gather*}
\Delta^{a,b}(e_{i,j}t^s)=e_{i,j}t^s\otimes1+1\otimes \delta(i,j\leq b)e_{i,j}t^s,
\end{gather*}
\begin{gather*}
\Delta^{a,b}(H_{i,0})=H_{i,0}\otimes1+1\otimes H_{i,0},\\
\Delta^{a,b}(X^\pm_{i,0})=X^\pm_{i,0}\otimes1+1\otimes X^\pm_{i,0},
\end{gather*}
\begin{align*}
\Delta^{a,b}(H_{i,1})&=\begin{cases}
(H_{0,1}+B_0)\otimes 1+1\otimes H_{0,1}+A_0-F_0\text{ if }i=0,\\
(H_{i,1}+B_i)\otimes 1+1\otimes H_{0,1}+A_i-F_i\text{ if }i\neq0,
\end{cases}\\
\Delta^{a,b}(X^+_{i,1})&=\begin{cases}
(X^+_{0,1}+B^+_0)\otimes 1+1\otimes X^+_{0,1}+A^+_0-F^+_0\text{ if }i=0,\\
(X^+_{i,1}+B^+_i)\otimes 1+1\otimes X^+_{0,1}+A^+_i-F^+_i\text{ if }i\neq0,
\end{cases}\\
\Delta^{a,b}(X^-_{i,1})&=\begin{cases}
(X^-_{0,1}+B^-_0)\otimes 1+1\otimes X^-_{0,1}+A^-_0-F^-_0\text{ if }i=0,\\
(X^-_{i,1}+B^-_i)\otimes 1+1\otimes X^-_{0,1}+A^-_i-F^-_i\text{ if }i\neq0,
\end{cases}
\end{align*}
where 
\begin{align*}
F_i&=\begin{cases}
\hbar\displaystyle\sum_{w\in\mathbb{Z}}\limits\displaystyle\displaystyle\sum_{v=n+1}^b e_{v,i}t^w\otimes e_{i,v}t^{-w}-\hbar\sum_{w\in\mathbb{Z}}\limits\displaystyle\sum_{v=n+1}^b e_{v,i+1}t^w\otimes e_{i+1,v}t^{-w}\text{ if }i\neq0,\\
\hbar\displaystyle\sum_{w\in\mathbb{Z}}\limits\displaystyle\sum_{v=n+1}^b e_{v,n}t^w\otimes e_{n,v}t^{-w}-\hbar\displaystyle\sum_{w\in\mathbb{Z}}\limits\displaystyle\sum_{v=n+1}^b e_{v,1}t^w\otimes e_{1,v}t^{-w}\text{ if }i=0.
\end{cases}\\
F^+_i&=\begin{cases}
\hbar\displaystyle\sum_{w\in\mathbb{Z}}\limits\sum_{u=n+1}^b e_{u,1}t^{-w}\otimes e_{n,u}t^{w+1}\text{ if }i=0,\\
\hbar\displaystyle\sum_{w\in\mathbb{Z}}\limits\sum_{u=n+1}^b e_{u,i+1}t^{-w}\otimes e_{i,u}t^{w}\text{ if }i\neq0,
\end{cases}\\
F^-_i&=\begin{cases}
\hbar\displaystyle\sum_{w\in\mathbb{Z}}\limits\sum_{u=n+1}^b e_{u,n}t^{-w}\otimes e_{1,u}t^{w-1}\text{ if }i=0,\\
\hbar\displaystyle\sum_{w\in\mathbb{Z}}\limits\sum_{u=n+1}^b e_{u,i}t^{-w}\otimes e_{i+1,u}t^{w}\text{ if }i\neq0,
\end{cases}\\
A_i&=\begin{cases}
-\hbar(e_{1,1}\otimes e_{n,n}+e_{n,n}\otimes e_{1,1})+\hbar(e_{n,n}-e_{1,1})\otimes c_b+\hbar c_a\otimes(e_{n,n}-e_{1,1})+\hbar c_a\otimes c_b\\
\quad+\hbar\displaystyle\sum_{s \geq 0} \limits\displaystyle\sum_{u=1}^{n}\limits (-e_{u,n}t^{-s-1}\otimes e_{n,u}t^{s+1}+e_{n,u}t^{-s}\otimes e_{u,n}t^s)\\
\quad-\hbar\displaystyle\sum_{s \geq 0}\displaystyle\sum_{u=1}^{n}\limits (-e_{u,1}t^{-s}\otimes e_{1,u}t^{s}+e_{1,u}t^{-s-1}\otimes e_{u,1}t^{s+1})\\
\qquad\qquad\qquad\qquad\qquad\qquad\qquad\qquad\qquad\qquad\qquad\qquad\qquad\qquad \text{ if }i=0,\\
-\hbar(e_{i,i}\otimes e_{i+1,i+1}+e_{i+1,i+1}\otimes e_{i,i})\\
\quad+\hbar\displaystyle\sum_{s \geq 0}  \limits\displaystyle\sum_{u=1}^{i}\limits (-e_{u,i}t^{-s-1}\otimes e_{i,u}t^{s+1}+e_{i,u}t^{-s}\otimes e_{u,i}t^s)\\
\quad+\hbar\displaystyle\sum_{s \geq 0} \limits\displaystyle\sum_{u=i+1}^{n}\limits (-e_{u,i}t^{-s}\otimes e_{i,u}t^{s}+e_{i,u}t^{-s-1}\otimes e_{u,i}t^{s+1})\\
\quad-\hbar\displaystyle\sum_{s \geq 0}\limits\displaystyle\sum_{u=1}^{i}\limits (-e_{u,i+1}t^{-s-1}\otimes e_{i+1,u}t^{s+1}+e_{i+1,u}t^{-s}\otimes e_{u,i+1}t^s)\\
\quad-\hbar\displaystyle\sum_{s \geq 0}\limits\displaystyle\sum_{u=i+1}^{n} \limits (-e_{u,i+1}t^{-s}\otimes e_{i+1,u}t^{s}+e_{i+1,u}t^{-s-1}\otimes e_{u,i+1}t^{s+1})\\
\qquad\qquad\qquad\qquad\qquad\qquad\qquad\qquad\qquad\qquad\qquad\qquad\qquad\qquad \text{ if }i\neq0,
\end{cases}\\
A^+_i&=\begin{cases}
\hbar c_a\otimes e_{n,1}t+\hbar\displaystyle\sum_{s \geq 0} \limits\displaystyle\sum_{u=1}^{n}\limits e_{u,1}t^{-s}\otimes e_{n,u}t^{s+1} -\hbar\sum_{s \geq 0} \limits\displaystyle\sum_{u=1}^{n}\limits e_{n,u}t^{-s}\otimes e_{u,1}t^{s+1}\\
\qquad\qquad\qquad\qquad\qquad\qquad\qquad\qquad\qquad\qquad\qquad\qquad\qquad\qquad \text{ if $i=0$},\\
\hbar\displaystyle\sum_{s \geq 0}\limits\displaystyle\sum_{u=1}^i\limits (-e_{u,i+1}t^{-s-1}\otimes e_{i,u}t^{s+1}+e_{i,u}t^{-s}\otimes e_{u,i+1}t^s)\\
+\hbar\displaystyle\sum_{s \geq 0}\limits\displaystyle\sum_{u=i+1}^{n}\limits (-e_{u,i+1}t^{-s}\otimes e_{i,u}t^{s}+e_{i,u}t^{-s-1}\otimes e_{u,i+1}t^{s+1})\\
\qquad\qquad\qquad\qquad\qquad\qquad\qquad\qquad\qquad\qquad\qquad\qquad\qquad\qquad \text{ if $i \neq 0$},
\end{cases}\\
A^-_i&=\begin{cases}
\hbar e_{1,n}t^{-1}\otimes c_b+\hbar\displaystyle\sum_{s \geq 0} \limits\displaystyle\sum_{u=1}^{n}\limits (-e_{u,n}t^{-s-1}\otimes e_{1,u}t^{s}+e_{1,u}t^{-s-1}\otimes e_{u,n}t^s)\\
\qquad\qquad\qquad\qquad\qquad\qquad\qquad\qquad\qquad\qquad\qquad\qquad\qquad\qquad \text{ if $i = 0$},\\
\hbar\displaystyle\sum_{s \geq 0}\limits\displaystyle\sum_{u=1}^i\limits (-e_{u,i}t^{-s-1}\otimes e_{i+1,u}t^{s+1}+e_{i+1,u}t^{-s}\otimes e_{u,i}t^s)\\
+\hbar\displaystyle\sum_{s \geq 0}\limits\displaystyle\sum_{u=i+1}^{n}\limits (-e_{u,i}t^{-s}\otimes e_{i+1,u}t^s+e_{i+1,u}t^{-s-1}\otimes e_{u,i}t^{s+1})\\
\qquad\qquad\qquad\qquad\qquad\qquad\qquad\qquad\qquad\qquad\qquad\qquad\qquad\qquad\text{ if $i \neq 0$}.
\end{cases}\\
B_i&=\begin{cases}
\hbar\displaystyle\sum_{s\geq0}\limits\displaystyle\sum_{u=b+1}^a\limits (e_{u,n}t^{-s-1}e_{n,u}t^{s+1}+e_{n,u}t^{-s}e_{u,n}t^{s})\\
\quad-\hbar\displaystyle\sum_{s\geq0}\limits\displaystyle\sum_{u=b+1}^a\limits (e_{u,1}t^{-s-1}e_{1,u}t^{s+1}+e_{1,u}t^{-s}e_{u,1}t^{s})\\
\quad-\hbar\displaystyle\sum_{w\leq m-n}\limits W^{(1)}_{w,w}+\hbar(a-b)e_{n,n}t+\hbar (a-b)c_a\\
\qquad\qquad\qquad\qquad\qquad\qquad\qquad\qquad\qquad\qquad\qquad\qquad\qquad\qquad \text{ if $i = 0$},\\
\hbar\displaystyle\sum_{s \geq 0}  \limits\displaystyle\sum_{u=b+1}^{a}\limits e_{u,i}t^{-s-1}e_{i,u}t^{s+1}-\hbar\displaystyle\sum_{s \geq 0}  \limits\displaystyle\sum_{u=b+1}^{a}\limits e_{u,i+1}t^{-s}e_{i+1,u}t^{s}\\
\qquad\qquad\qquad\qquad\qquad\qquad\qquad\qquad\qquad\qquad\qquad\qquad\qquad\qquad\text{ if }i\neq0,
\end{cases}\\
B^+_i&=\begin{cases}
\hbar\displaystyle\sum_{s\geq0}\limits\displaystyle\sum_{u=b+1}^a\limits (e_{u,1}t^{-s-1}e_{n,u}t^{s+2}+e_{n,u}t^{1-s}e_{u,1}t^{s})\\
\qquad\qquad\qquad\qquad\qquad\qquad\qquad\qquad\qquad\qquad\qquad\qquad\qquad\qquad \text{ if $i = 0$},\\
\hbar\displaystyle\sum_{s\geq0}\limits\displaystyle\sum_{u=b+1}^a\limits (e_{u,i+1}t^{-s-1}e_{i,u}t^{s+1}+e_{i,u}t^{-s}e_{u,i+1}t^{s})\\
\qquad\qquad\qquad\qquad\qquad\qquad\qquad\qquad\qquad\qquad\qquad\qquad\qquad\qquad \text{ if $i \neq 0$},
\end{cases}\\
B^-_i&=\begin{cases}
\hbar\displaystyle\sum_{s\geq0}\limits\sum_{u=n+1}^a\limits (e_{u,n}t^{-s-1}e_{1,u}t^{s}+e_{1,u}t^{-1-s}e_{u,n}t^{s})+\hbar(a-b)e_{1,n}t^{-1}\\
\qquad\qquad\qquad\qquad\qquad\qquad\qquad\qquad\qquad\qquad\qquad\qquad\qquad\qquad \text{ if $i = 0$},\\
\hbar\displaystyle\sum_{s\geq0}\limits\sum_{u=b+1}^a\limits (e_{u,i}t^{-s-1}e_{i+1,u}t^{s}+e_{i+1,u}t^{-1-s}e_{u,i}t^{s})\\
\qquad\qquad\qquad\qquad\qquad\qquad\qquad\qquad\qquad\qquad\qquad\qquad\qquad\qquad\text{ if $i \neq 0$}.
\end{cases}
\end{align*}
\end{Theorem}
The proof of Theorem~\ref{Coproduct} will be written in the appendix.
It is enough to show the compatibility with \eqref{Eq2.1}-\eqref{Eq2.10} and \eqref{Eq2.11}-\eqref{Eq2.20}. We only prove that $\Delta^{a,b}$ is compatible with $[H_{i,1},H_{j,1}]=0$, \eqref{Eq2.11} and \eqref{Eq2.12}. We show the compatibility with \eqref{Eq2.11} and \eqref{Eq2.12} in appendix A and the one with $[H_{i,1},H_{j,1}]=0$ in appendix B.
We can prove the other compatibilities in a similar way. 
\begin{Remark}
In the case when $a=b=n$, we have the following relation by the definition of $\Delta^{a,b}$;

\begin{equation*}
(\Psi_1\otimes\Psi_1)\circ\Delta=\Delta^{n,n}\circ\Psi.
\end{equation*}
By this remark, we find that $\Delta^{a,b}$ is the natural extension of $\Delta$.
\end{Remark}
\section{$W$-algebras of type $A$}
We fix some notations for vertex algebras. For a vertex algebra $V$, we denote the generating field associated with $v\in V$ by $v(z)=\displaystyle\sum_{n\in\mathbb{Z}}\limits v_{(n)}z^{-n-1}$. We also denote the OPE of $V$ by
\begin{equation*}
u(z)v(w)\sim\displaystyle\sum_{s\geq0}\limits \dfrac{(u_{(s)}v)(w)}{(z-w)^{s+1}}
\end{equation*}
for all $u, v\in V$. We denote the vacuum vector (resp.\ the translation operator) by $|0\rangle$ (resp.\ $\partial$).

We set
\begin{equation*}
N=\displaystyle\sum_{i=1}^lq_i,\qquad q_1\geq q_2\geq\cdots\geq q_l.
\end{equation*}
We set a basis of $\mathfrak{gl}(N)$ as $\mathfrak{gl}(N)=\displaystyle\bigoplus_{1\leq i,j\leq N}\limits\mathbb{C}e_{i,j}$. 
We also fix an inner product of $\mathfrak{gl}(N)$ determined by
\begin{equation*}
(e_{i,j}|e_{p,q})=k\delta_{i,q}\delta_{p,j}+\delta_{i,j}\delta_{p,q}.
\end{equation*}
\begin{gather*}
\col(i)=s\text{ if }\sum_{j=1}^{s-1}q_j<i\leq\sum_{i=1}^sq_j,\\
\row(i)=i-\sum_{j=1}^{\col(i)-1}q_j.
\end{gather*}
For all $1\leq i,j\leq N$, we take $1\leq \hat{i},\tilde{i}\leq N$ as
\begin{gather*}
\col(\hat{i})=\col(i)+1,\row(\hat{i})=\row(i),\\
\col(\tilde{j})=\col(j)-1,\row(\tilde{j})=\row(j).
\end{gather*}

We set a nilpotent element $f$ as 
\begin{equation*}
f=\sum_{1\leq j\leq N}\limits e_{\hat{j},j}.
\end{equation*}
We take an $\mathfrak{sl}(2)$-triple $(x,e,f)$, that is,
\begin{equation*}
x=[e,f],\ [x,e]=e,\ [x,f]=-f
\end{equation*}
satisfying that
\begin{equation*}
\{y\in\mathfrak{gl}(N)|[x,y]=sy\}=\bigoplus_{\col(j)-\col(i)=s}\limits e_{i,j}.
\end{equation*}
We can set the grading of $\mathfrak{gl}(N)$ by $\text{deg}(e_{i,j})=\col(j)-\col(i)$ (see Section 7 in \cite{BK}). We will denote $\bigoplus_{\col(j)-\col(i)=s}\limits e_{i,j}$ by $\mathfrak{g}_s$.

We consider two vertex algebras. The first one is the universal affine vertex algebra associated with a Lie subalgebra
\begin{align*}
\mathfrak{b}&=\bigoplus_{\substack{1\leq i,j\leq N\\\col(i)\geq\col(j)}}\limits \mathbb{C}e_{i,j}\subset\mathfrak{gl}(N)
\end{align*}
and its inner product
\begin{equation*}
\kappa(e_{i,j},e_{p,q})=\alpha_{\col(i)}\delta_{i,q}\delta_{p,j}+\delta_{i,j}\delta_{p,q},
\end{equation*}
where $\alpha_i=k+N-q_i$. 

The second one is the universal affine vertex algebra associated with a Lie superalgebra $\mathfrak{a}=\mathfrak{b}\oplus\displaystyle\bigoplus_{\substack{1\leq i,j\leq N\\\col(i)>\col(j)}}\limits\mathbb{C}\psi_{i,j}$  with the following commutator relations;
\begin{align*}
[e_{i,j},\psi_{p,q}]&=\delta_{j,p}\psi_{i,q}-\delta_{i,q}\psi_{p,j},\\
[\psi_{i,j},\psi_{p,q}]&=0,
\end{align*}
where $e_{i,j}$ is an even element and $\psi_{i,j}$ is an odd element.
We set the inner product on $\mathfrak{a}$ such that
\begin{gather*}
\widetilde{\kappa}(e_{i,j},e_{p,q})=\kappa(e_{i,j},e_{p,q}),\qquad\widetilde{\kappa}(e_{i,j},\psi_{p,q})=\widetilde{\kappa}(\psi_{i,j},\psi_{p,q})=0.
\end{gather*}
By the definition of $V^{\widetilde{\kappa}}(\mathfrak{a})$ and $V^\kappa(\mathfrak{b})$, $V^{\widetilde{\kappa}}(\mathfrak{a})$ contains $V^\kappa(\mathfrak{b})$.

By the PBW theorem, we can identify $V^{\widetilde{\kappa}}(\mathfrak{a})$ (resp. $V^\kappa(\mathfrak{b})$) with $U(\mathfrak{a}[t^{-1}])$ (resp. $U(\mathfrak{b}[t^{-1}])$). In order to simplify the notation, here after, we denote the generating field $(ut^{-1})(z)$ as $u(z)$. By the definition of $V^{\widetilde{\kappa}}(\mathfrak{a})$, generating fields $u(z)$ and $v(z)$ satisfy the OPE
\begin{gather}
u(z)v(w)\sim\dfrac{[u,v](w)}{z-w}+\dfrac{\widetilde{\kappa}(u,v)}{(z-w)^2}\label{OPE1}
\end{gather}
for all $u,v\in\mathfrak{a}$. 

For all $u\in \mathfrak{a}$, let $u[-s]$ be $ut^{-s}$. In this section, we regard $V^{\widetilde{\kappa}}(\mathfrak{a})$ (resp.\ $V^\kappa(\mathfrak{b})$) as a non-associative superalgebra whose product $\cdot$ is defined by
\begin{equation*}
u[-w]\cdot v[-s]=(u[-w])_{(-1)}v[-s].
\end{equation*}
We sometimes omit $\cdot$ and in order to simplify the notation. By \cite{KW1} and \cite{KW2}, a $W$-algebra $\mathcal{W}^k(\mathfrak{gl}(N),f)$ can be realized as a subalgebra of $V^\kappa(\mathfrak{b})$.

Let us define an odd differential $d_0 \colon V^{\kappa}(\mathfrak{b})\to V^{\widetilde{\kappa}}(\mathfrak{a})$ determined by
\begin{gather}
d_01=0,\\
[d_0,\partial]=0,\label{ee5800}
\end{gather}
\begin{align}
[d_0,e_{i,j}[-1]]
&=\sum_{\substack{\col(i)>\col(r)\geq\col(j)}}\limits e_{r,j}[-1]\psi_{i,r}[-1]-\sum_{\substack{\col(j)<\col(r)\leq\col(i)}}\limits \psi_{r,j}[-1]e_{i,r}[-1]\nonumber\\
&\quad+\delta(\col(i)>\col(j))\alpha_{\col(i)}\psi_{i,j}[-2]+\psi_{\hat{i},j}[-1]-\psi_{i,\tilde{j}}[-1].\label{ee1}
\end{align}
By using Theorem 2.4 in \cite{KRW}, we can define the $W$-algebra $\mathcal{W}^k(\mathfrak{g},f)$ as follows.
\begin{Definition}\label{T125}
The $W$-algebra $\mathcal{W}^k(\mathfrak{gl}(N),f)$ is the vertex subalgebra of $V^\kappa(\mathfrak{b})$ defined by
\begin{equation*}
\mathcal{W}^k(\mathfrak{gl}(N),f)=\{y\in V^\kappa(\mathfrak{b})\subset V^{\widetilde{\kappa}}(\mathfrak{a})\mid d_0(y)=0\}.
\end{equation*}
\end{Definition}

We construct two kinds of elements $W^{(1)}_{i,j}$ and $W^{(2)}_{i,j}$. 

\begin{Theorem}\label{Generators}
Let us set
\begin{align*}
W^{(1)}_{p,q}&=\sum_{\substack{1\leq i,j\leq N,\\\row(i)=p,\row(j)=q,\\\col(i)=\col(j)}}e_{i,j}[-1]\text{ for }q_l<p=q\leq q_1\text{ or }1\leq p,q\leq q_l,\\
W^{(2)}_{p,q}&=\sum_{\substack{\col(i)=\col(j)+1\\\row(i)=p,\row(j)=q}}e_{i,j}[-1]-\sum_{\substack{\col(i)=\col(j)\\\row(i)=p,\row(j)=q}}\gamma_{\col(i)}e_{i,j}[-2]\\
&\quad+\sum_{\substack{\col(u)=\col(j)<\col(i)=\col(v)\\\row(u)=\row(v)\leq q_l\\\row(i)=p,\row(j)=q}}\limits e_{u,j}[-1]e_{i,v}[-1]-\sum_{\substack{\col(u)=\col(j)\geq\col(i)=\col(v)\\\row(u)=\row(v)>q_l\\\row(i)=p,\row(j)=q}}\limits e_{u,j}[-1]e_{i,v}[-1]\\
&\qquad\qquad\qquad\qquad\qquad\qquad\qquad\qquad\qquad\qquad\qquad\qquad\qquad\text{ \qquad\qquad\qquad for }p,q\leq q_l,
\end{align*}
where
\begin{gather*}
\gamma_a=\sum_{u=a+1}^{l}\limits \alpha_{u}.
\end{gather*}
Then, the $W$-algebra $\mathcal{W}^k(\mathfrak{gl}(N),f)$ contains $W^{(1)}_{p,q}$ and $W^{(2)}_{p,q}$.
\end{Theorem}
\begin{proof}
It is enough to show that $d_0(W^{(r)}_{p,q})=0$.
First, we show the case when $r=1$.
By \eqref{ee1}, if $\col(i)=\col(j)$, we obtain
\begin{align}
[d_0,e_{i,j}[-1]]
&=\psi_{\widehat{i},j}[-1]-\psi_{i,\widetilde{j}}[-1].\label{ee307}
\end{align}
By \eqref{ee307}, we obtain
\begin{align}
d_0(W^{(1)}_{p,q})&=\sum_{\substack{1\leq i,j\leq N,\\\row(i)=p,\row(j)=q,\\\col(i)=\col(j)}}(\psi_{\widehat{i},j}[-1]-\psi_{i,\widetilde{j}}[-1])\nonumber\\
&=\sum_{\substack{1\leq i,j\leq N,\\\row(i)=p,\row(j)=q,\\\col(i)=\col(j)}}\psi_{\widehat{i},j}[-1]-\sum_{\substack{1\leq i,j\leq N,\\\row(i)=p,\row(j)=q,\\\col(i)=\col(j)}}\psi_{i,\widetilde{j}}[-1].\label{ee380}
\end{align}
In the case when $q_l<p=q\geq q_1\text{ or }p,q\leq q_l$, we can rewrite the second term of \eqref{ee380} as
\begin{equation*}
-\sum_{\substack{1\leq x,y\leq N,\\\row(x)=p,\row(y)=q,\\\col(x)=\col(y)}}\psi_{\widehat{x},y}[-1]
\end{equation*}
by setting $\widehat{x}=i,y=\widetilde{j}$. Thus, we obtain $d_0(W^{(1)}_{i,j})=0$.

Next, we show the case when $r=2$. 
If $\col(i)=\col(j)+1=2$, by \eqref{ee1}, we also have
\begin{align}
&\quad[d_0,e_{i,j}[-1]]\nonumber\\
&=\sum_{\substack{\col(r)=\col(j)}}\limits e_{r,j}[-1]\psi_{i,r}[-1]-\sum_{\substack{\col(r)=\col(i)}}\limits \psi_{r,j}[-1]e_{i,r}[-1]\nonumber\\
&\quad+\alpha_{\col(i)}\psi_{i,j}[-2]+\psi_{\hat{i},j}[-1]-\psi_{i,\tilde{j}}[-1].\label{ee2}
\end{align}
By the definition of $W^{(2)}_{i,j}$, we can rewrite $d_0(W^{(2)}_{p,q})$ as
\begin{align}
&\sum_{\substack{\col(i)=\col(j)+1\\\row(i)=p,\row(j)=q}}d_0(e_{i,j}[-1])-\sum_{\substack{\col(i)=\col(j)\\\row(i)=p,\row(j)=q}}\gamma_{\col(i)}d_0(e_{i,j}[-2])\nonumber\\
&\quad+\sum_{\substack{\col(u)=\col(j)<\col(i)=\col(v)\\\row(u)=\row(v)\leq q_l\\\row(i)=p,\row(j)=q}}\limits d_0(e_{u,j}[-1])e_{i,v}[-1]\nonumber\\
&\quad+\sum_{\substack{\col(u)=\col(j)<\col(i)=\col(v)\\\row(u)=\row(v)\leq q_l\\\row(i)=p,\row(j)=q}}\limits e_{u,j}[-1]d_0(e_{i,v}[-1])\nonumber\\
&\quad-\sum_{\substack{\col(u)=\col(j)\geq\col(i)=\col(v)\\\row(u)=\row(v)>q_l\\\row(i)=p,\row(j)=q}}\limits d_0(e_{u,j}[-1])e_{i,v}[-1]\nonumber\\
&\quad-\sum_{\substack{\col(u)=\col(j)\geq\col(i)=\col(v)\\\row(u)=\row(v)>q_l\\\row(i)=p,\row(j)=q}}\limits e_{u,j}[-1]d_0(e_{i,v}[-1]).\label{ee3}
\end{align}
By \eqref{ee2}, we obtain
\begin{align}
&\quad\text{the first term of \eqref{ee3}}\nonumber\\
&=\sum_{\substack{\col(i)=\col(j)+1\\\row(i)=p,\row(j)=q}}\sum_{\substack{\col(r)=\col(j)}}\limits e_{r,j}[-1]\psi_{i,r}[-1]-\sum_{\substack{\col(i)=\col(j)+1\\\row(i)=p,\row(j)=q}}\sum_{\substack{\col(r)=\col(i)}}\limits \psi_{r,j}[-1]e_{i,r}[-1]\nonumber\\
&\quad+\sum_{\substack{\col(i)=\col(j)+1\\\row(i)=p,\row(j)=q}}\alpha_{\col(i)}\psi_{i,j}[-2]+\sum_{\substack{\col(i)=\col(j)+1\\\row(i)=p,\row(j)=q}}(\psi_{\hat{i},j}[-1]-\psi_{i,\tilde{j}}[-1]).\label{ee5}
\end{align}
Similarly to the proof of $d_0(W^{(1)}_{i,j})=0$, we find that the last term of the right hand side of \eqref{ee5} is equal to zero. Then, we have
\begin{align}
&\quad\text{the first term of \eqref{ee3}}\nonumber\\
&=\sum_{\substack{\col(i)=\col(j)+1\\\row(i)=p,\row(j)=q}}\sum_{\substack{\col(r)=\col(j)}}\limits e_{r,j}[-1]\psi_{i,r}[-1]-\sum_{\substack{\col(i)=\col(j)+1\\\row(i)=p,\row(j)=q}}\sum_{\substack{\col(r)=\col(i)}}\limits \psi_{r,j}[-1]e_{i,r}[-1]\nonumber\\
&\quad+\sum_{\substack{\col(i)=\col(j)+1\\\row(i)=p,\row(j)=q}}\alpha_{\col(i)}\psi_{i,j}[-2].\label{ee5.1}
\end{align}

By \eqref{ee307} and \eqref{ee5800}, we obtain
\begin{align}
&\quad\text{the second term of \eqref{ee3}}\nonumber\\
&=-\sum_{\substack{\col(i)=\col(j)\\\row(i)=p,\row(j)=q}}\gamma_{\col(i)}(\psi_{\widehat{i},j}[-2]-\psi_{i,\widetilde{j}}[-2])\nonumber\\
&=\sum_{\substack{\col(i)=\col(j)\\\row(i)=p,\row(j)=q}}(\gamma_{\col(\hat{i})}-\gamma_{\col(i)})\psi_{\widehat{i},j}[-2]\nonumber\\
&=-\sum_{\substack{\col(i)=\col(j)\\\row(i)=p,\row(j)=q}}\limits\alpha_{\col(\hat{i})}\psi_{\hat{i},j}[-2].\label{ee6}
\end{align}
By \eqref{ee307}, we obtain
\begin{align}
&\quad\text{the third term of the right hand side of \eqref{ee3}}\nonumber\\
&=\sum_{\substack{\col(u)=\col(j)<\col(i)=\col(v)\\\row(u)=\row(v)\leq q_l\\\row(i)=p,\row(j)=q}}\limits (\psi_{\hat{u},j}[-1]-\psi_{u,\tilde{j}}[-1])e_{i,v}[-1]\nonumber\\
&=\sum_{\substack{\col(u)+1=\col(j)+1=\col(i)=\col(v)\\\row(u)=\row(v)\leq q_l\\\row(i)=p,\row(j)=q}}\limits \psi_{\hat{u},j}[-1]e_{i,v}[-1]\label{ee7},\\
&\quad\text{the 4-th term of the right hand side of \eqref{ee3}}\nonumber\\
&=\sum_{\substack{\col(u)=\col(j)<\col(i)=\col(v)\\\row(u)=\row(v)\leq q_l\\\row(i)=p,\row(j)=q}}\limits e_{u,j}[-1](\psi_{\hat{i},v}[-1]-\psi_{i,\tilde{v}}[-1])\nonumber\\
&=-\sum_{\substack{\col(u)+1=\col(j)+1=\col(i)=\col(v)\\\row(u)=\row(v)\leq q_l\\\row(i)=p,\row(j)=q}}\limits e_{u,j}[-1]\psi_{i,\tilde{v}}[-1],\label{ee8}\\
&\quad\text{the 5-th term of the right hand side of \eqref{ee3}}\nonumber\\
&=-\sum_{\substack{\col(u)=\col(j)\geq\col(i)=\col(v)\\q_l<\row(u)=\row(v)\leq q_{\col(j)}\\\row(i)=p,\row(j)=q}}\limits (\psi_{\hat{u},j}[-1]-\psi_{u,\tilde{j}}[-1])e_{i,v}[-1]\nonumber\\
&=\sum_{\substack{\col(u)=\col(j)=\col(i)+1=\col(v)+1\\q_l<\row(u)=\row(v)\leq q_{\col(j)}\\\row(i)=p,\row(j)=q}}\limits\psi_{u,\tilde{j}}[-1]e_{i,v}[-1],\label{ee9}\\
&\quad\text{the 6-th term of the right hand side of \eqref{ee3}}\nonumber\\
&=-\sum_{\substack{\col(u)=\col(j)\geq\col(i)=\col(v)\\q_l<\row(u)=\row(v)\leq q_{\col(j)}\\\row(i)=p,\row(j)=q}}\limits e_{u,j}[-1](\psi_{\hat{i},v}[-1]-\psi_{i,\tilde{v}}[-1])\nonumber\\
&=-\sum_{\substack{\col(u)=\col(j)=\col(i)+1=\col(v)+1\\q_l<\row(u)=\row(v)\leq q_{\col(j)}\\\row(i)=p,\row(j)=q}}\limits e_{u,j}[-1]\psi_{\hat{i},v}[-1].\label{ee10}
\end{align}
Here after, in order to simplify the notation, let us denote the $i$-th term of (the number of the equation) by 
$\text{(the number of the equation)}_i$. By a direct computation, we obtain
\begin{gather*}
\eqref{ee5}_1+\eqref{ee8}+\eqref{ee10}=0,\\
\eqref{ee5}_2+\eqref{ee7}+\eqref{ee9}=0,\\
\eqref{ee5}_3+\eqref{ee6}=0.
\end{gather*}
Then, adding \eqref{ee5}-\eqref{ee10}, we obtain $d_0(W^{(2)}_{i,j})=0$.
\end{proof}
\begin{Remark}
We have already considered the case when $l=2,q_1=m,q_2=n$ in \cite{U5}. In \cite{U5},
we use the different notations about $\col(i)$ and $\row(i)$ as follows;
\begin{gather*}
\col(i)=\begin{cases}
1&\text{ if }i\leq m,\\
2&\text{ if }i>m,
\end{cases}\qquad\row(i)=\begin{cases}
i&\text{ if }i\leq m,\\
i-n&\text{ if }i>m.
\end{cases}
\end{gather*}
In \cite{U5}, we give the strong generators of $\mathcal{W}^k(\mathfrak{gl}(m+n),f)$ as follows;
\begin{equation*}
\{W^{(1)}_{i,j}\mid i\leq m-n,1\leq j\leq m\text{ or }i,j>m-n\},\\
\{W^{(2)}_{i,j}\mid i>m-n\}.
\end{equation*}
The elements $W^{(1)}_{i,j}$ and $W^{(2)}_{i,j}$ in Theorem~\ref{Generators} are corresponding to the elements $W^{(1)}_{m-i+1,m-j+1}$ and $W^{(2)}_{m-i+1,m-j+1}$ in \cite{U5}.
\end{Remark}
Let us set the centralizer
\begin{equation*}
\mathfrak{g}_0^f=\{y\in\mathfrak{g}_0|[f,y]=0\}.
\end{equation*}
By Theorem~2.4 in \cite{KRW}, there exists an embedding from the$\mathfrak{g}_0^f$ to $\mathcal{W}^k(\mathfrak{gl}(N),f)$. We note that $W^{(1)}_{p,q}$ is corresponding to $\sum_{\substack{1\leq i,j\leq N,\\\row(i)=p,\row(j)=q,\\\col(i)=\col(j)}}e_{i,j}$ for $q_l<p=q\leq q_1\text{ or }1\leq p,q\leq q_l$.

\section{The universal enveloping algebra of $\mathcal{W}^k(\mathfrak{gl}(N),f)$}
Let us recall the definition of a universal enveloping algebra of a vertex algebra in the sense of \cite{FZ} and \cite{MNT}.
For any vertex algebra $V$, let $L(V)$ be the Borchards Lie algebra, that is,
\begin{align}
 L(V)=V{\otimes}\mathbb{C}[t,t^{-1}]/\text{Im}(\partial\otimes\id +\id\otimes\frac{d}{d t})\label{844},
\end{align}
where the commutation relation is given by
\begin{align*}
 [ut^a,vt^b]=\sum_{r\geq 0}\begin{pmatrix} a\\r\end{pmatrix}(u_{(r)}v)t^{a+b-r}
\end{align*}
for all $u,v\in V$ and $a,b\in \mathbb{Z}$. Now, we define the universal enveloping algebra of $V$.
\begin{Definition}[Section~6 in \cite{MNT}]\label{Defi}
We set $\mathcal{U}(V)$ as the quotient algebra of the standard degreewise completion of the universal enveloping algebra of $L(V)$ by the completion of the two-sided ideal generated by
\begin{gather}
(u_{(a)}v)t^b-\sum_{i\geq 0}
\begin{pmatrix}
 a\\i
\end{pmatrix}
(-1)^i(ut^{a-i}vt^{b+i}-(-1)^avt^{a+b-i}ut^{i}),\label{241}\\
|0\rangle t^{-1}-1.
\end{gather}
We call $\mathcal{U}(V)$ the universal enveloping algebra of $V$.
\end{Definition}
In the last of this section, we will consider the universal enveloping algebra of 
$\mathcal{W}^k(\mathfrak{gl}(N),f)$.
The projection map from $\mathfrak{gl}(N)$ to $\bigotimes_{i=1}^l\limits \mathfrak{gl}(q_i)$ induces the injective homomorphism called the Miura map (see \cite{KW1})
\begin{align*}
\mu\colon \mathcal{W}^k(\mathfrak{gl}(N),f)\to V^{\kappa}(\bigotimes_{i=1}^l\limits \mathfrak{gl}(q_i)).
\end{align*}

Let $e^{(r)}_{i,j}t^s\in\bigotimes_{1\leq i\leq l}\limits U(\widehat{\mathfrak{gl}}(q_i))$ be $1^{\otimes r-1}\otimes e^{(r)}_{i,j}t^s\otimes 1^{\otimes l-r}$. Let us set the degree of $\bigotimes_{1\leq i\leq l}\limits U(\widehat{\mathfrak{gl}}(q_i))$ by
\begin{equation*}
\text{deg}(e^{(r)}_{i,j}t^s)=s.
\end{equation*}
Induced by the Miura map $\mu$, we obtain
\begin{equation*}
\widetilde{\mu}\colon \mathcal{U}(\mathcal{W}^{k}(\mathfrak{gl}(N),f))\to {\widehat{\bigoplus}}_{1\leq i\leq l}U(\widehat{\mathfrak{gl}}(q_i)),
\end{equation*}
where ${\widehat{\bigotimes}}_{1\leq i\leq l}U(\widehat{\mathfrak{gl}}(q_i))$ is the standard degreewise completion of $\bigotimes_{1\leq i\leq l}U(\widehat{\mathfrak{gl}}(q_i))$. 

By the definition of $W^{(1)}_{i,j}$ and $W^{(2)}_{i,j}$, we have
\begin{align}
\widetilde{\mu}(W^{(1)}_{i,j}t^s)&=\sum_{1\leq r\leq l}\limits e^{(r)}_{i,j}t^s,\label{W1}\\
\widetilde{\mu}(W^{(2)}_{i,j}t^s)&=\sum_{r=1}^ns\gamma_re^{(r)}_{i,j}t^{s-1}+\sum_{s\in\mathbb{Z}}\limits\sum_{r_1<r_2}\sum_{1\leq u\leq q_l}\limits e^{(r_1)}_{u,j}t^{-s}e^{(r_2)}_{i,u}t^s\nonumber\\
&\quad-\sum_{s\in\mathbb{Z}}\limits\sum_{r_1<r_2}\sum_{q_l<u\leq q_{r_2}}\limits e^{(r_1)}_{i,u}t^{-s}e^{(r_2)}_{u,j}t^s\nonumber\\
&\quad-\sum_{s\geq0}\limits\sum_{r\geq0}\limits\sum_{1\leq u \leq q_r}\limits (e^{(r)}_{u,j}t^{-s-1}e^{(r)}_{i,u}t^{s+1}+e^{(r)}_{i,u}t^{-s}e^{(r)}_{u,j}t^{s}).\label{W2}
\end{align}
Since the Miura map is injective (see \cite{F}, \cite{A3}), $\widetilde{\mu}$ is injective.
\section{Guay's affine Yangians and non-rectangular $W$-algebras}

Similarly to $Y^a_{\hbar,\ve-(a-b)\hbar}(\widehat{\mathfrak{sl}}(n))\widetilde{\otimes}Y^{b}_{\hbar,\ve}(\widehat{\mathfrak{sl}}(n))$, we define
\begin{equation*}
Y^{q_g}_{\hbar,\ve-(q_g-q_l)\hbar}(\widehat{\mathfrak{sl}}(n))\widetilde{\otimes}Y^{q_{g+1}}_{\hbar,\ve-(q_{g+1}-q_l)\hbar}(\widehat{\mathfrak{sl}}(n))\widetilde{\otimes}\cdots\widetilde{\otimes}Y^{q_{l-1}}_{\hbar,\ve-(q_{l-1}-q_l)\hbar}(\widehat{\mathfrak{sl}}(n))\widetilde{\otimes}Y^{q_{l}}_{\hbar,\ve}(\widehat{\mathfrak{sl}}(n)).
\end{equation*}
We denote this algebra by $\widetilde{\otimes}_{i=g}^lY^{q_i}_{\hbar,\ve-(q_i-q_l)\hbar}(\widehat{\mathfrak{sl}}(n))$.
By Theorem~\ref{Coproduct}, $\Delta^{q_{g-1},q_{g}}$ naturally induces the homomorphism
\begin{equation*}
\widetilde{\otimes}_{i=g+1}^lY^{q_i}_{\hbar,\ve-(q_i-q_l)\hbar}(\widehat{\mathfrak{sl}}(n))\to\widetilde{\otimes}_{i=g}^lY^{q_i}_{\hbar,\ve-(q_i-q_l)\hbar}(\widehat{\mathfrak{sl}}(n)).
\end{equation*}
We denote this homomorphism by $\Delta^{q_{g-1},q_{g}}\otimes\id^{\otimes l-g}$.
By the definition of $\Delta^{q_{g-1},q_{g}}\otimes\id^{\otimes l-g}$, we have a homomorphism
\begin{equation*}
\Delta^l\colon Y_{\hbar,\ve}(\widehat{\mathfrak{sl}}(n))\to \widetilde{\otimes}_{i=1}^lY^{q_i}_{\hbar,\ve-(q_i-q_l)\hbar}(\widehat{\mathfrak{sl}}(n))
\end{equation*}
determined by
\begin{equation*}
\Delta^l=(\Delta^{q_1,q_2}\otimes\id^{l-2})\circ(\Delta^{q_2,q_3}\otimes\id^{l-3})\circ\cdots\circ(\Delta^{q_{l-2},q_{l-1}}\otimes\id)\circ\Delta^{q_{l-1},q_l}\circ\Psi_1.
\end{equation*}
By Theorem~\eqref{evaluation}, we have a homomorphism
\begin{equation*}
\widetilde{\ev}_{\hbar,\ve-(q_i-q_l)\hbar}\colon Y^{q_i}_{\hbar,\ve-(q_i-q_l)\hbar}(\widehat{\mathfrak{sl}}(n))\to U(\widehat{\mathfrak{gl}}(q_i))_{\text{comp}}
\end{equation*}
under the assumption that
\begin{equation*}
c_{q_i}=-\dfrac{\ve-(q_i-q_l)\hbar}{\hbar}.
\end{equation*}
By the definition of $\widetilde{\otimes}_{i=g}^lY^{q_i}_{\hbar,\ve-(q_i-q_l)\hbar}(\widehat{\mathfrak{sl}}(n))$, we find that
\begin{align*}
\widetilde{\ev}_{\hbar,\ve-(q_1-q_l)\hbar}^{-\hbar\sum_{v=2}^l\alpha_v}\otimes\widetilde{\ev}_{\hbar,\ve-(q_2-q_l)\hbar}^{-\hbar\sum_{v=3}^l\alpha_v}\otimes\cdots\otimes\widetilde{\ev}_{\hbar,\ve-(q_{l-1}-q_l)\hbar}^{-\hbar\alpha_l}\circ\widetilde{\ev}_{\hbar,\ve}^{0}.
\end{align*}
induces the homomorphism
\begin{equation*}
\ev^l\colon \widetilde{\otimes}_{i=1}^lY^{q_i}_{\hbar,\ve-(q_i-q_l)\hbar}(\widehat{\mathfrak{sl}}(n))\to U(\widehat{\mathfrak{gl}}(q_1))\widehat{\otimes}\cdots\widehat{\otimes} U(\widehat{\mathfrak{gl}}(q_l)).
\end{equation*}
Here after, we sometimes denote $q_l$ by $n$.
\begin{Theorem}
Suppose that $n\geq 3$ and $-\dfrac{\ve}{\hbar}=k+N$.
There exists an algebra homomorphism
\begin{equation*}
\Phi\colon Y_{\hbar,\ve}(\widehat{\mathfrak{sl}}(n))\to \mathcal{U}(\mathfrak{gl}(N),f))
\end{equation*}
determined by
\begin{gather*}
\Phi(H_{i,0})=\begin{cases}
W^{(1)}_{n,n}-W^{(1)}_{n,n}+\displaystyle\sum_{v=1}^{l}\limits\alpha_v&\text{ if }i=0,\\
W^{(1)}_{i,i}-W^{(1)}_{i+1,i+1i}&\text{ if }i\neq 0,
\end{cases}\\
\Phi(X^+_{i,0})=\begin{cases}
W^{(1)}_{n,1}t&\text{ if }i=0,\\
W^{(1)}_{i,i+1}&\text{ if }i\neq0,
\end{cases}
\quad \Phi(X^-_{i,0})=\begin{cases}
W^{(1)}_{1,n}t^{-1}&\text{ if }i=0,\\
W^{(1)}_{i+1,i}&\text{ if }i\neq0,
\end{cases}
\end{gather*}
\begin{align*}
\Phi(H_{i,1})&=\begin{cases}
-\hbar(W^{(2)}_{n,n}t-W^{(2)}_{1,1}t)-\hbar(\displaystyle\sum_{v=1}^{l-1}\alpha_v)W^{(1)}_{n,n}\\
\quad-\hbar(\displaystyle\sum_{v=1}^{l-1}\alpha_v)(\displaystyle\sum_{v=1}^{l}\alpha_v)+\hbar(\displaystyle\sum_{v=1}^{l}\alpha_v)\Phi(H_{0,0}) -\hbar W^{(1)}_{n,n} (W^{(1)}_{1,1}-(\displaystyle\sum_{v=1}^{l}\alpha_v))\\
\quad-\hbar\displaystyle\sum_{w\leq m-n}\limits W^{(1)}_{w,w}+\hbar\displaystyle\sum_{s \geq 0} \limits\displaystyle\sum_{u=1}^{n}\limits W^{(1)}_{n,u}t^{-s} W^{(1)}_{u,n}t^s-\hbar\displaystyle\sum_{s \geq 0}\displaystyle\sum_{u=1}^{n}\limits W^{(1)}_{1,u}t^{-s-1} W^{(1)}_{u,1}t^{s+1}\\
\qquad\qquad\qquad\qquad\qquad\qquad\qquad\qquad\qquad\qquad\qquad\qquad\qquad\qquad \text{ if }i=0,\\
-\hbar(W^{(2)}_{i,i}t-W^{(2)}_{i+1,i+1}t)-\dfrac{i}{2}\hbar\Phi(H_{i,0})+\hbar W^{(1)}_{i,i}W^{(1)}_{i+1,i+1}\\
\quad+\hbar\displaystyle\sum_{s \geq 0}  \limits\displaystyle\sum_{u=1}^{i}\limits W^{(1)}_{i,u}t^{-s}W^{(1)}_{u,i}t^s+\hbar\displaystyle\sum_{s \geq 0} \limits\displaystyle\sum_{u=i+1}^{n}\limits W^{(1)}_{i,u}t^{-s-1} W^{(1)}_{u,i}t^{s+1}\\
\quad-\hbar\displaystyle\sum_{s \geq 0}\limits\displaystyle\sum_{u=1}^{i}\limits W^{(1)}_{i+1,u}t^{-s} W^{(1)}_{u,i+1}t^s-\hbar\displaystyle\sum_{s \geq 0}\limits\displaystyle\sum_{u=i+1}^{n} \limits W^{(1)}_{i+1,u}t^{-s-1} W^{(1)}_{u,i+1}t^{s+1}\\
\qquad\qquad\qquad\qquad\qquad\qquad\qquad\qquad\qquad\qquad\qquad\qquad\qquad\qquad \text{ if }i\neq0,
\end{cases}
\end{align*}
\begin{align*}
\Phi(X^+_{i,1})&=\begin{cases}
-\hbar W^{(2)}_{n,1}t^2+\hbar
(\displaystyle\sum_{v=1}^{l}\alpha_v)\Phi(X_{0,0}^{+})+\hbar\displaystyle\sum_{s \geq 0} \limits\displaystyle\sum_{u=1}^{n}\limits W^{(1)}_{n,u}t^{-s} W^{(1)}_{u,1}t^{s+1}\\
\qquad\qquad\qquad\qquad\qquad\qquad\qquad\qquad\qquad\qquad\qquad\qquad\qquad\qquad \text{ if $i = 0$},\\
-\hbar W^{(2)}_{i,i+1}t-\dfrac{i}{2}\hbar\Phi(X_{i,0}^{+})\\
\quad+\hbar\displaystyle\sum_{s \geq 0}\limits\displaystyle\sum_{u=1}^i\limits W^{(1)}_{i,u}t^{-s} W^{(1)}_{u,i+1}t^s+\hbar\displaystyle\sum_{s \geq 0}\limits\displaystyle\sum_{u=i+1}^{n}\limits W^{(1)}_{i,u}t^{-s-1} W^{(1)}_{u,i+1}t^{s+1}\\
\qquad\qquad\qquad\qquad\qquad\qquad\qquad\qquad\qquad\qquad\qquad\qquad\qquad\qquad \text{ if $i \neq 0$},
\end{cases}
\end{align*}
\begin{align*}
\Phi(X^-_{i,1})&=\begin{cases}
-\hbar W^{(2)}_{1,n}-\hbar(\displaystyle\sum_{v=1}^{l-1}\alpha_v)W^{(1)}_{1,n}t^{-1}-\hbar(\displaystyle\sum_{r=1}^l\limits\alpha_r)\Phi(X_{0,0}^{-})+\hbar\displaystyle\sum_{s \geq 0} \limits\displaystyle\sum_{u=1}^{n}\limits W^{(1)}_{1,u}t^{-s-1} W^{(1)}_{u,n}t^s\\
\qquad\qquad\qquad\qquad\qquad\qquad\qquad\qquad\qquad\qquad\qquad\qquad\qquad\qquad \text{ if $i = 0$},\\
-\hbar W^{(2)}_{i+1,i}t-\dfrac{i}{2}\hbar\Phi(X_{i,0}^{-})\\
\quad+\hbar\displaystyle\sum_{s \geq 0}\limits\displaystyle\sum_{u=1}^i\limits W^{(1)}_{i+1,u}t^{-s} W^{(1)}_{u,i}t^s+\hbar\displaystyle\sum_{s \geq 0}\limits\displaystyle\sum_{u=i+1}^{n}\limits W^{(1)}_{i+1,u}t^{-s-1} W^{(1)}_{u,i}t^{s+1} \\
\qquad\qquad\qquad\qquad\qquad\qquad\qquad\qquad\qquad\qquad\qquad\qquad\qquad\qquad\text{ if $i \neq 0$}.
\end{cases}
\end{align*}
\end{Theorem}
\begin{proof}
Since $\widetilde{\mu}$ is injective, it is enough to show the following relation;
\begin{gather*}
\ev^l\circ\Delta^l(A_{i,r})=\widetilde{\mu}\circ\Phi(A_{i,r})\text{ for all }r=0,1\text{ and }A=H,X^\pm.
\end{gather*}
It is enough to show that 
\begin{gather*}
\ev^l\circ\Delta^l(A_{i,r})=\widetilde{\mu}\circ\Phi(A_{i,r})\text{ for all }r=0,1\text{ and }A=H,X^\pm.
\end{gather*}
We only show the case when $i=0,r=1,A=X^\pm$. Other cases are proven in a similar way.
First, we show that
\begin{gather}
\ev^l\circ\Delta^l(X^+_{0,1})=\widetilde{\mu}\circ\Phi(X^+_{0,1}).\label{gather1}
\end{gather}
By \eqref{W1} and \eqref{W2}, the right hand side of \eqref{gather1} is equal to
\begin{align}
&-\hbar\sum_{w\in\mathbb{Z}}\limits\sum_{r_1<r_2}\limits\sum_{1\leq u\leq n}\limits e^{(r_1)}_{u,1}t^{-w+1}e^{(r_2)}_{n,u}t^w+\hbar\sum_{w\in\mathbb{Z}}\limits\sum_{r_1<r_2}\limits\sum_{n<u\leq q_{r_2}}\limits e^{(r_1)}_{n,u}t^we^{(r_2)}_{u,1}t^{-w+1}\nonumber\\
&\quad+\hbar\sum_{s\geq0}\limits\sum_{1\leq r\leq l}\limits\sum_{n<u\leq q_r}\limits (e^{(r)}_{u,1}t^{-s-1}e^{(r)}_{n,u}t^{s+2}+e^{(r)}_{n,u}t^{-s+1}e^{(r)}_{u,1}t^s)\nonumber\\
&\quad-2\hbar\sum_{a=1}^l\limits(\sum_{r=a+1}^l\alpha_r)e^{(a)}_{n,1}t
+\hbar\sum_{a=1}^l\limits(\sum_{r=1}^l\limits\alpha_r)e^{(a)}_{n,1}t+\hbar\displaystyle\sum_{s \geq 0} \limits\sum_{r=1}^l\limits\displaystyle\sum_{u=1}^{n}\limits e^{(r)}_{n,u}t^{-s} e^{(r)}_{u,1}t^{s+1}\nonumber\\
&\quad+\hbar\displaystyle\sum_{s \geq 0} \limits\sum_{r_1<r_2}\limits\displaystyle\sum_{u=1}^{n}\limits e^{(r_1)}_{n,u}t^{-s} e^{(r_2)}_{u,1}t^{s+1}+\hbar\displaystyle\sum_{s \geq 0} \limits\sum_{r_1<r_2}\limits\displaystyle\sum_{u=1}^{n}\limits e^{(r_1)}_{u,1}t^{s+1}e^{(r_2)}_{n,u}t^{-s},\label{align1}
\end{align}
where the first four terms are derived from $-\hbar W^{(1)}_{n,1}t^2$.
By Theorem~\ref{evaluation} and \ref{Coproduct}, the left hand side of \eqref{gather1} is equal to
\begin{align}
&\hbar\sum_{r=1}^l\limits \alpha_r e^{(r)}_{n,1}t+\hbar\sum_{r=1}^l\limits\sum_{u=1}^n\limits e^{(r)}_{n,u}t^{-s}e^{(r)}_{u,1}t^{s+1}-\hbar\sum_{a=1}^l\limits(\sum_{r=a+1}^l\alpha_r)e^{(a)}_{n,1}t\nonumber\\
&\quad+\hbar\sum_{a=1}^l\limits(\sum_{r=1}^{a-1}\alpha_r)e^{(a)}_{n,1}t+\hbar\displaystyle\sum_{s \geq 0} \limits\displaystyle\sum_{u=1}^{n}\limits\sum_{r_1<r_2}\limits (-e^{(r_1)}_{u,n}t^{-s}e^{(r_2)}_{1,u}t^{s+1}+e^{(r_1)}_{1,u}t^{-s}e^{(r_2)}_{u,n}t^{s+1})\nonumber\\
&\quad+\hbar\sum_{w\in\mathbb{Z}}\limits\sum_{r_1<r_2}\limits\sum_{u=n+1}^{q_{r_2}} e^{(r_1)}_{n,u}t^{-w}e^{(r_2)}_{u,1}t^{w+1}\nonumber\\
&\quad+\hbar\sum_{r=1}^l\sum_{s\geq0}\limits\sum_{u=n+1}^a\limits (e^{(r)}_{u,1}t^{-s-1}e^{(r)}_{n,u}t^{s+2}+e^{(r)}_{n,u}t^{1-s}e^{(r)}_{u,1}t^{s}),\label{align3}
\end{align}
where $\eqref{align3}_1$, $\eqref{align3}_2$, $\eqref{align3}_3$ are deduced from the evaluation map, $\eqref{align3}_4$, $\eqref{align3}_5$ are deduced from $A_i$, and other terms are deduced from $B_i$ and $F_i$.
Since the relations
\begin{gather*}
\eqref{align3}_1+\eqref{align3}_3+\eqref{align3}_4=\eqref{align1}_4+\eqref{align1}_5,\\
\eqref{align3}_2=\eqref{align1}_6,\\
\eqref{align3}_5=\eqref{align1}_1+\eqref{align1}_6+\eqref{align1}_7,\\
\eqref{align3}_6=\eqref{align1}_2,\\
\eqref{align3}_7=\eqref{align1}_3
\end{gather*}
hold by a direct computation, we obtain \eqref{gather1}.

Next, we show that
\begin{gather}
\ev^l\circ\Delta^l(X^-_{0,1})=\widetilde{\mu}\circ\Phi(X^-_{0,1}).\label{gather2}
\end{gather}
By \eqref{W1} and \eqref{W2}, the right hand side of \eqref{gather2} is equal to
\begin{align}
&-\hbar\sum_{w\in\mathbb{Z}}\limits\sum_{r_1<r_2}\limits\sum_{1\leq u\leq n}\limits e^{(r_1)}_{u,n}t^{-w-1}e^{(r_2)}_{1,u}t^w\nonumber\\
&\quad+\hbar\sum_{w\in\mathbb{Z}}\limits\sum_{r_1<r_2}\limits\sum_{n<u\leq q_{r_2}}\limits e^{(r_1)}_{1,u}t^we^{(r_2)}_{u,n}t^{-w-1}\nonumber\\
&\quad+\hbar\sum_{s\geq0}\limits\sum_{1\leq r\leq l}\limits\sum_{n<u\leq q_r}\limits (e^{(r)}_{u,n}t^{-s-1}e^{(r)}_{1,u}t^s+e^{(r)}_{1,u}t^{-s-1}e^{(r)}_{u,n}t^s)-\hbar\sum_{a=1}^l\limits(\sum_{r=1}^{l-1}\limits\alpha_r)e^{(a)}_{1,n}t^{-1}\nonumber\\
&\quad+\hbar\sum_{a=1}^l\limits(\sum_{r=1}^l\limits\alpha_r)e^{(a)}_{1,n}t^{-1}+\hbar\displaystyle\sum_{s \geq 0} \limits\sum_{r=1}^l\limits\displaystyle\sum_{u=1}^{n}\limits e^{(r)}_{1,u}t^{-s-1} e^{(r)}_{u,n}t^s\nonumber\\
&\quad+\hbar\displaystyle\sum_{s \geq 0} \limits\sum_{r_1<r_2}\limits\displaystyle\sum_{u=1}^{n}\limits e^{(r_1)}_{1,u}t^{-s-1} e^{(r_2)}_{u,n}t^s+\hbar\displaystyle\sum_{s \geq 0} \limits\sum_{r_1<r_2}\limits\displaystyle\sum_{u=1}^{n}\limits e^{(r_1)}_{u,n}t^se^{(r_2)}_{1,u}t^{-s-1},\label{align4}
\end{align}
where the first three terms are derived from $-\hbar W^{(2)}_{1,n}$.
By Theorem~\ref{evaluation} and \ref{Coproduct}, the left hand side of \eqref{gather2} is equal to
\begin{align}
&\hbar\sum_{r=1}^l\limits \alpha_r e^{(r)}_{1,n}t^{-1}+\hbar\sum_{r=1}^l\limits\sum_{u=1}^n\limits e^{(r)}_{1,u}t^{-s-1}e^{(r)}_{u,n}t^s-\hbar\sum_{a=1}^l\limits(\sum_{r=a+1}^l\alpha_r)e^{(a)}_{1,n}t^{-1}\nonumber\\
&\quad+\hbar\sum_{a=1}^l\limits(\sum_{r=a+1}^l\alpha_r)e^{(a)}_{1,n}t^{-1}+\hbar\displaystyle\sum_{s \geq 0} \limits\displaystyle\sum_{u=1}^{n}\limits\sum_{r_1<r_2}\limits (-e^{(r_1)}_{u,n}t^{-s-1}e^{(r_2)}_{1,u}t^{s}+e^{(r_1)}_{1,u}t^{-s-1}e^{(r_2)}_{u,n}t^s)\nonumber\\
&\quad+\hbar\sum_{s\geq0}\limits\sum_{r_1<r_2}\limits\sum_{u=n+1}^{q_{r_2}} e^{(r_1)}_{1,u}t^{-w-1}e^{(r_2)}_{u,n}t^{w}\nonumber\\
&\quad+\hbar\sum_{r=1}^l\sum_{s\geq0}\limits\sum_{u=n+1}^a\limits (e^{(r)}_{u,n}t^{-s-1}e^{(r)}_{1,u}t^{s}+e^{(r)}_{1,u}t^{-s-1}e_{u,n}t^{s})\nonumber\\
&\quad-\hbar\sum_{a=1}^l(\alpha_a-\alpha_l) e^{(a)}_{1,n}t^{-1},\label{align5}
\end{align}
where $\eqref{align4}_1$, $\eqref{align4}_2$, $\eqref{align4}_3$ are deduced from the evaluation map, $\eqref{align4}_4$, $\eqref{align4}_5$ are deduced from $A_i$, and other terms are deduced from $B_i$ and $F_i$.
Since the relations
\begin{gather*}
\eqref{align5}_1+\eqref{align5}_3+\eqref{align5}_4+\eqref{align5}_8=\eqref{align4}_4+\eqref{align4}_5,\\
\eqref{align5}_2=\eqref{align4}_6,\\
\eqref{align5}_5=\eqref{align4}_1+\eqref{align4}_8+\eqref{align4}_7,\\
\eqref{align5}_6=\eqref{align4}_2,\\
\eqref{align5}_7=\eqref{align4}_3
\end{gather*}
hold by a direct computation, we obtain \eqref{gather2}.
\end{proof}
\begin{Remark}\label{Rmk1}
In Theorem~5.1 in \cite{U5}, we have constructed an algebra homomorphism
\begin{equation*}
\widetilde{\Phi}\colon Y_{\hbar,\ve}(\widehat{\mathfrak{sl}}(n))\to\mathcal{U}(\mathcal{W}^k(\mathfrak{gl}(m+n),f))
\end{equation*}
in the case when $q_1=m,q_2=n$. By the definition of $\Phi$ and $\widetilde{\Phi}$, we find that $\widetilde{\Phi}$ is different from $\Phi$. However, in the same computation to the one of the proof of Theorem 5.1 in \cite{U5}, we can prove that $\Phi$ is compatible with the defining relations \eqref{Eq2.1}-\eqref{Eq2.10}. 
\end{Remark}
\appendix
\section{The proof of the compatibility with \eqref{Eq2.11} and \eqref{Eq2.12}}
Take $n<u\leq b$. By the definition of $\Delta^{a,b}$, we have
\begin{align*}
&\quad[\Delta^{a,b}(H_{i,1}),\Delta^{a,b}(e_{u,j}t^x)]\\
&=[H_{i,1},e_{u,j}t^x]\otimes 1+1\otimes[H_{i,1},e_{u,j}t^x]-[F_i,\square(e_{u,j}t^x)]+[A_i,\square(e_{u,j}t^x)]+[B_i,\square(e_{u,j}t^x)],
\end{align*}
where $\square(y)=y\otimes1+1\otimes y$. Thus, it is enough to show that
\begin{equation*}
(\Delta^{a,b}-\square)([H_{i,1}e_{u,j}t^x])=-[F_i,\square(e_{u,j}t^x)]+[A_i,\square(e_{u,j}t^x)]+[B_i,\square(e_{u,j}t^x)].
\end{equation*}

First, let us compute $[F_i,\square(e_{u,j}t^x)]$. By a direct computation, we obtain
\begin{align}
&\quad[\hbar\sum_{w\in\mathbb{Z}}\limits\sum_{v=n+1}^b e_{v,i}t^w\otimes e_{i,v}t^{-w},\square(e_{u,j}t^x)]\nonumber\\
&=\hbar\sum_{w\in\mathbb{Z}}\limits e_{u,i}t^w\otimes e_{i,j}t^{x-w}-\delta_{i,j}\hbar\sum_{w\in\mathbb{Z}}\limits\sum_{v=n+1}^b e_{v,i}t^w\otimes e_{u,v}t^{x-w}\nonumber\\
&\quad-\alpha_2\delta_{i,j}\hbar xe_{u,i}t^x\otimes 1.\label{align100}
\end{align}
Here after, we denote by $(\text{equation number})_{p,q}$ the formula that substitutes $i=p,j=q$ for the right hand side of $(\text{equation number})$.
By the definition of $F_i$, we find that
\begin{align*}
[F_i,\square(e_{u,j}t^x)]&=\eqref{align100}_{i,j}-\eqref{align100}_{i+1,j}.
\end{align*}
Next, let us compute $[A_i,\square(e_{u,j}t^x)]$. By a direct computation, we obtain
\begin{align}
&\quad[-\hbar(e_{i,i}\otimes e_{i+1,i+1}+e_{i+1,i+1}\otimes e_{i,i}),\square(e_{u,j}t^x)]\nonumber\\
&=\hbar\delta_{j,i+1} e_{i,i}\otimes e_{u,j}t^x+\hbar\delta_{i,j}e_{u,j}t^x\otimes e_{i+1,i+1}\nonumber\\
&\quad+\hbar\delta_{i,j}e_{i+1,i+1}\otimes e_{u,j}t^x+\hbar\delta_{j,i+1}e_{u,j}t^x\otimes e_{i,i},\label{align100.5}
\end{align}
\begin{align}
&\quad[\hbar\displaystyle\sum_{s \geq 0}  \limits\displaystyle\sum_{k=1}^{i}\limits (-e_{k,i}t^{-s-1}\otimes e_{i,k}t^{s+1}+e_{i,k}t^{-s}\otimes e_{k,i}t^s),\square(e_{u,j}t^x)]\nonumber\\
&=\hbar\displaystyle\sum_{s \geq 0}\limits\delta(j\leq i) e_{u,i}t^{a-s-1}\otimes e_{i,j}t^{s+1}+\delta_{i,j}\hbar\displaystyle\sum_{s \geq 0}  \limits\displaystyle\sum_{k=1}^{i}\limits e_{k,i}t^{-s-1}\otimes e_{u,k}t^{s+x+1}\nonumber\\
&\quad-\delta_{i,j}\hbar\displaystyle\sum_{s \geq 0}  \limits\displaystyle\sum_{k=1}^{i}\limits e_{u,k}t^{x-s}\otimes e_{k,i}t^s-\hbar\displaystyle\sum_{s \geq 0}  \limits\delta(j\leq i) e_{i,j}t^{-s}\otimes e_{u,i}t^{s+x},\label{align101}
\end{align}
\begin{align}
&\quad[\hbar\displaystyle\sum_{s \geq 0}  \limits\displaystyle\sum_{k=i+1}^{n}\limits (-e_{k,i}t^{-s}\otimes e_{i,k}t^{s}+e_{i,k}t^{-s-1}\otimes e_{k,i}t^{s+1}),\square(e_{u,j}t^x)]\nonumber\\
&=\hbar\displaystyle\sum_{s \geq 0}\limits\delta(j> i) e_{u,i}t^{x-s}\otimes e_{i,j}t^{s}+\delta_{i,j}\hbar\displaystyle\sum_{s \geq 0}  \limits\displaystyle\sum_{k=i+1}^{n}\limits e_{k,i}t^{-s}\otimes e_{u,k}t^{s+x}\nonumber\\
&\quad-\delta_{i,j}\hbar\displaystyle\sum_{s \geq 0}  \limits\displaystyle\sum_{k=i+1}^{n}\limits e_{u,k}t^{x-s-1}\otimes e_{k,i}t^{s+1}-\hbar\displaystyle\sum_{s \geq 0}\limits\delta(j>i) e_{i,j}t^{-s-1}\otimes e_{u,i}t^{s+x+1},\label{align102}
\end{align}
\begin{align}
&\quad-[\hbar\displaystyle\sum_{s \geq 0}  \limits\displaystyle\sum_{k=1}^{i}\limits (-e_{k,i+1}t^{-s-1}\otimes e_{i+1,k}t^{s+1}+e_{i+1,k}t^{-s}\otimes e_{k,i}t^s),\square(e_{u,j}t^x)]\nonumber\\
&=-\hbar\displaystyle\sum_{s \geq 0}\limits\delta(j\leq i) e_{u,i+1}t^{x-s-1}\otimes e_{i+1,j}t^{s+1}-\delta_{i+1,j}\hbar\displaystyle\sum_{s \geq 0}  \limits\displaystyle\sum_{k=1}^{i}\limits e_{k,i+1}t^{-s-1}\otimes e_{u,k}t^{s+x+1}\nonumber\\
&\quad+\delta_{i+1,j}\hbar\displaystyle\sum_{s \geq 0}  \limits\displaystyle\sum_{k=1}^{i}\limits e_{u,k}t^{x-s}\otimes e_{k,i}t^s+\hbar\displaystyle\sum_{s \geq 0}  \limits\delta(j\leq i) e_{i+1,j}t^{-s}\otimes e_{u,i+1}t^{s+x},\label{align104}
\end{align}
\begin{align}
&\quad-[\hbar\displaystyle\sum_{s \geq 0}  \limits\displaystyle\sum_{k=i+1}^{n}\limits (-e_{k,i}t^{-s}\otimes e_{i,k}t^{s}+e_{i,k}t^{-s-1}\otimes e_{k,i}t^{s+1}),\square(e_{u,j}t^x)]\nonumber\\
&=-\hbar\displaystyle\sum_{s \geq 0}\limits\delta(j> i) e_{u,i+1}t^{x-s}\otimes e_{i+1,j}t^{s}-\delta_{i,j}\hbar\displaystyle\sum_{s \geq 0}  \limits\displaystyle\sum_{k=i+1}^{n}\limits e_{k,i+1}t^{-s}\otimes e_{u,k}t^{s+x}\nonumber\\
&\quad+\delta_{i,j}\hbar\displaystyle\sum_{s \geq 0}  \limits\displaystyle\sum_{k=i+1}^{n}\limits e_{u,k}t^{x-s-1}\otimes e_{k,i+1}t^{s+1}+\hbar\displaystyle\sum_{s \geq 0}\limits\delta(j>i) e_{i+1,j}t^{-s-1}\otimes e_{u,i+1}t^{s+x+1}.\label{align105}
\end{align}
By the definition of $A_i$, we find that
\begin{align*}
[A_i,\square(e_{u,j}t^x)]&=\eqref{align100.5}_{i,j}+\eqref{align101}_{i,j}+\eqref{align102}_{i,j}+\eqref{align104}_{i,j}+\eqref{align105}_{i,j}.
\end{align*}
Finally, let us compute $[B_i\otimes1,\square(e_{u,j}t^a)]$. By a direct computation, we obtain
\begin{align}
&\quad[\hbar\sum_{s\geq0}\limits\sum_{v=b+1}^a\limits(e_{v,i}t^{-s-1}e_{i,v}t^{s+1}+e_{i,v}t^{-s}e_{v,i}t^s), e_{u,j}t^x]\nonumber\\
&=-\hbar\delta_{i,j}\sum_{s\geq0}\limits\sum_{v=b+1}^a\limits(e_{v,i}t^{-s-1}e_{u,v}t^{x+s+1}+e_{u,v}t^{x-s}e_{v,i}t^s).\label{align103}
\end{align}
By the definition of $B_i$, we have
\begin{align*}
[B_i\otimes1,\square(e_{u,j}t^x)]&=\eqref{align103}_{i,j}-\eqref{align103}_{i+1,j}
\end{align*}

Here after, we denote by $(\text{equation number})_{p,q,m}$ $m$-th term of the right hand side of the formula that substitutes $i=p,j=q$ for the right hand side of $(\text{equation number})$.

Since we obtain
\begin{align*}
&\quad-\eqref{align100}_{i,j,1}+\eqref{align101}_{i,j,1}+\eqref{align102}_{i,j,1}\\
&=-\hbar\displaystyle\sum_{s \geq 0}\limits\delta(j\leq i) e_{u,i}t^{s+x}\otimes e_{i,j}t^{-s}-\hbar\displaystyle\sum_{s \geq 0}\limits\delta(j> i) e_{u,i}t^{x+s+1}\otimes e_{i,j}t^{-s-1}.
\end{align*}
by a direct computation, we have
\begin{align}
&\quad-\eqref{align100}_{i,j,1}+\eqref{align101}_{i,j,1}+\eqref{align102}_{i,j,1}+\eqref{align101}_{i,j,4}+\eqref{align102}_{i,j,4}\nonumber\\
&=-(\Delta^{a,b}-\square)(\hbar\displaystyle\sum_{s \geq 0}\limits\delta(j\leq i)e_{i,j}t^{-s}e_{u,i}t^{s+x})-(\Delta^{a,b}-\square)(\hbar\displaystyle\sum_{s \geq 0}\limits\delta(j> i) e_{i,k}t^{-s-1}e_{u,i}t^{x+s+1}).\label{align902}
\end{align}
Since we also obtain
\begin{align*}
&\quad\eqref{align100}_{i+1,j,1}+\eqref{align104}_{i,j,1}+\eqref{align105}_{i,j,1}\\
&=\hbar\displaystyle\sum_{s \geq 0}\limits\delta(j\leq i) e_{u,i+1}t^{s+x}\otimes e_{i+1,j}t^{-s}+\hbar\displaystyle\sum_{s \geq 0}\limits\delta(j> i) e_{u,i+1}t^{x+s+1}\otimes e_{i+1,j}t^{-s-1}.
\end{align*}
by a direct computation, we find that
\begin{align}
&\quad\eqref{align100}_{i+1,j,1}+\eqref{align104}_{i,j,1}+\eqref{align105}_{i,j,1}+\eqref{align104}_{i,j,4}+\eqref{align105}_{i,j,4}\nonumber\\
&=(\Delta^{a,b}-\square)(\hbar\displaystyle\sum_{s \geq 0}\limits\delta(j\leq i) e_{i+1,j}t^{-s}e_{u,i+1}t^{s+x})\nonumber\\
&\quad+(\Delta^{a,b}-\square)(\hbar\displaystyle\sum_{s \geq 0}\limits\delta(j> i) e_{i+1,j}t^{-s-1}e_{u,i+1}t^{x+s+1}).\label{align903}
\end{align}
First, let us show the compatibility with \eqref{Eq2.11}.
In the case when $i\neq j,j+1$, we obtain
\begin{align*}
&\quad[\Delta^{a,b}(H_{i,1}),\Delta^{a,b}(e_{u,j}t^a)]\\
&=-\eqref{align100}_{i,j,1}+\eqref{align101}_{i,j,1}+\eqref{align102}_{i,j,1}+\eqref{align101}_{i,j,4}+\eqref{align102}_{i,j,4}\nonumber\\
&\quad+\eqref{align100}_{i+1,j,1}+\eqref{align104}_{i,j,1}+\eqref{align105}_{i,j,1}+\eqref{align104}_{i,j,4}+\eqref{align105}_{i,j,4}.
\end{align*}
Thus, by \eqref{align902} and \eqref{align903}, we find that the compatibility with \eqref{Eq2.11}.

Next, we show the compatibility with \eqref{Eq2.12}. We write down \eqref{Eq2.12} as follows;
\begin{align*}
&\quad[H_{i-1,1},e_{u,i}t^x]+[H_{i,1},e_{u,i}t^x]\nonumber\\
&=\dfrac{\hbar}{2}e_{u,i}t^x+\hbar e_{i-1,i-1}e_{u,i}t^x+\hbar e_{u,i}t^xe_{i+1,i+1}\nonumber\\
&\quad-\hbar\displaystyle\sum_{s \geq 0} \limits e_{i-1,i}t^{-s-1}e_{u,i-1}t^{s+x+1}+\hbar\displaystyle\sum_{s \geq 0}\limits e_{i+1,i}t^{-s}e_{u,i+1}t^{s+w}\nonumber\\
&\quad-\hbar e_{u,i}t^{x}e_{i,i}-\hbar e_{i,i}e_{u,i}t^{x}.
\end{align*}
By the definition of $\Delta^{a,b}$, we obtain
\begin{align*}
&\quad[\Delta^{a,b}(H_{i-1,1}),\Delta^{a,b}(e_{u,i}t^x)]+[\Delta^{a,b}(H_{i,1}),\Delta^{a,b}(e_{u,i}t^x)]\\
&=\eqref{align100}_{i-1,i}-\eqref{align100}_{i,i}+\eqref{align100}_{i,i}-\eqref{align100}_{i+1,i}\\
&\quad+\eqref{align100.5}_{i-1,i}+\eqref{align101}_{i-1,i}+\eqref{align102}_{i-1,i}+\eqref{align104}_{i-1,i}+\eqref{align105}_{i-1,i}\\
&\quad+\eqref{align100.5}_{i,i}+\eqref{align101}_{i,i}+\eqref{align102}_{i,i}+\eqref{align104}_{i,i}+\eqref{align105}_{i,i}\\
&\quad+\eqref{align103}_{i-1,i}-\eqref{align103}_{i,i}+\eqref{align103}_{i,i}-\eqref{align103}_{i+1,i}.
\end{align*}
By a direct computation, we obtain
\begin{gather*}
\eqref{align100}_{i-1,i,2}-\eqref{align100}_{i,i,2}=0,\\
\eqref{align100}_{i-1,i,3}-\eqref{align100}_{i,i,3}=0,\\
\eqref{align103}_{i-1,i}-\eqref{align103}_{i,i}+\eqref{align103}_{i,i}-\eqref{align103}_{i+1,i}=0.
\end{gather*}
By a direct computation, we obtain
\begin{align}
\eqref{align100.5}_{i-1,i,1}+\eqref{align100.5}_{i,i,2}
&=\hbar(\Delta^{a,b}-\square)e_{i-1,i-1}e_{u,i}t^x+\hbar(\Delta^{a,b}-\square)e_{u,i}t^xe_{i+1,i+1}.\label{align906}
\end{align}
By using
\begin{align*}
\eqref{align104}_{i-1,i,2}-\eqref{align101}_{i,i,2}
&=\hbar\displaystyle\sum_{s \geq 0}  \limits e_{i,i}t^{-s-1}\otimes e_{u,i}t^{s+x+1},\nonumber\\
\eqref{align105}_{i-1,i,2}-\eqref{align102}_{i,i,2}
&=-\hbar\displaystyle\sum_{s \geq 0}  \limits e_{i,i}t^{-s}\otimes e_{u,i}t^{s+x},
\end{align*}
we obtain
\begin{equation}
\eqref{align104}_{i-1,i,2}-\eqref{align101}_{i,i,2}+\eqref{align102}_{i-1,i,2}-\eqref{align105}_{i,i,2}=-\hbar e_{i,i}\otimes e_{u,i}t^{x}.\label{align904}
\end{equation}
By using
\begin{align*}
\eqref{align104}_{i-1,i,3}-\eqref{align101}_{i,i,3}
&=-\hbar\displaystyle\sum_{s \geq 0}  \limits e_{u,i}t^{x-s}\otimes e_{i,i}t^s,\nonumber\\
\eqref{align105}_{i-1,i,3}-\eqref{align102}_{i,i,3}
&=\hbar\displaystyle\sum_{s \geq 0}  \limits e_{u,i}t^{x-s-1}\otimes e_{i,i}t^{s+1},
\end{align*}
we obtain
\begin{equation}
\eqref{align104}_{i-1,i,3}-\eqref{align101}_{i,i,3}+\eqref{align105}_{i-1,i,3}-\eqref{align102}_{i,i,3}=-\hbar e_{u,i}t^{x}\otimes e_{i,i}.\label{align905}
\end{equation}
By \eqref{align904} and \eqref{align905}, we have
\begin{align}
&\quad\eqref{align101}_{i-1,i,2}-\eqref{align101}_{i,i,2}+\eqref{align102}_{i-1,i,2}-\eqref{align102}_{i,i,2}\nonumber\\
&\qquad+\eqref{align101}_{i-1,i,3}-\eqref{align101}_{i,i,3}+\eqref{align102}_{i-1,i,3}-\eqref{align102}_{i,i,3}\nonumber\\
&=-(\Delta^{a,b}-\square)(\hbar e_{u,i}t^{x}e_{i,i}).\label{align907}
\end{align}
By a direct computation, we obtain
\begin{align}
\eqref{align902}_{i,i}+\eqref{align903}_{i-1,i}&=-(\Delta^{a,b}-\square)(\hbar e_{i,i}e_{u,i}t^{x}).\label{align908}
\end{align}
By \eqref{align902}, \eqref{align903}, \eqref{align906}, \eqref{align905}, \eqref{align907} and \eqref{align908}, we find the compatibility with \eqref{Eq2.12}.

\section{The proof of compatibility with $[H_{i,1},H_{j,1}]=0$}
By the definition of $A_i$, $F_i$ and $B_i$, we find that
\begin{align}
&\quad[\Delta^{a,b}(H_{i,1}),\Delta^{a,b}(H_{j,1})]\nonumber\\
&=[H_{i,1}+B_i,H_{j,1}+B_j]\otimes 1+[(H_{i,1}+B_i)\otimes 1,A_j]-[(H_{j,1}+B_j)\otimes 1,A_i]\nonumber\\
&\quad+[1\otimes H_{i,1},A_j]-[1\otimes H_{j,1},A_i]+[A_i,A_j]\nonumber\\
&\quad-[(H_{i,1}+B_i)\otimes 1,F_j]+[(H_{j,1}+B_j)\otimes 1,F_i]\nonumber\\
&\quad-[1\otimes H_{i,1},F_j]+[1\otimes H_{j,1},F_i]-[A_i,F_j]+[A_j,F_i]+[F_i,F_j].\label{DD15}
\end{align}
By a direct computation, we obtain 
\begin{equation}
[F_i,F_j]=0.\label{FF15}
\end{equation}

By a similar proof of Theorem 5.2 in \cite{GNW}, we have
\begin{equation}
[H_{i,1}\otimes 1,A_j]-[H_{j,1}\otimes 1,A_i]+[1\otimes H_{i,1},A_j]-[1\otimes H_{j,1},A_i]+[A_i,A_j]=0.\label{HA15}
\end{equation}

By \eqref{DD15}, \eqref{FF15} and \eqref{HA15}, it is enough to show the following two lemmas.
\begin{Lemma}\label{Lem1}
The following equation holds;
\begin{gather}
[H_{i,1}+B_i,H_{j,1}+B_j]\otimes 1+[B_i\otimes 1,A_j]-[B_j\otimes 1,A_i].\label{Lemeq1}
\end{gather}
\end{Lemma}
\begin{Lemma}\label{Lem2}
The following equation holds;
\begin{equation}
[(H_{i,1}+B_i)\otimes 1,F_j]-[(H_{j,1}+B_j)\otimes 1,F_i]+[1\otimes H_{i,1},F_j]-[1\otimes H_{j,1},F_i]+[A_i,F_j]-[A_j,F_i]=0.\label{Lemeq2}
\end{equation}
\end{Lemma}
\subsection{The proof of Lemma~\ref{Lem1}}
In this subsection, we prove Lemma~\ref{Lem1}. Let us consider the first term of the left hand side of \eqref{HA15}. By the defining relation $[H_{i,1},H_{j,1}]=0$, we obtain
\begin{align*}
&\quad\text{the first term of the left hand side of \eqref{HA15}}\\
&=[H_{i,1},B_j]-[H_{j,1},B_i]+[B_i,B_j].
\end{align*}
By the defining relations \eqref{Eq2.11}-\eqref{Eq2.20} and the form of $B_i$, it is no problem to assume that
\begin{gather*}
[H_{i,1},e_{v,j}t^w]=[\ev_{\hbar,\ve-(a-b)\hbar}(H_{i,1}),e_{v,j}t^w],\\
[H_{i,1},e_{j,v}t^w]=[\ev_{\hbar,\ve-(a-b)\hbar}(H_{i,1}),e_{j,v}t^w]
\end{gather*}
hold in $Y^a_{\hbar,\ve-(a-b)\hbar}(\widehat{\mathfrak{sl}}(n))$. Thus, it is enough to show that
\begin{align}
&[\ev_{\hbar,\ve-(a-b)\hbar}(H_{i,1})+B_i,\ev_{\hbar,\ve-(a-b)\hbar}(H_{j,1})+B_j]\otimes 1+[B_i\otimes 1,A_j]-[B_j\otimes 1,A_i]\label{Lemeq100}
\end{align}
is equal to zero.
By Theorem 5.1 in \cite{U5} and Remark~\ref{Rmk1}, \eqref{Lemeq100} holds in the case when $b=n$. Comparing the two cases when $b=n$ and $b>n$, the difference of \eqref{Lemeq100} comes from the difference of the inner form. In the computation of \eqref{Lemeq100}, the terms affected by the inner product are 
\begin{align*}
&\displaystyle\sum_{s_1,s_2\geq0}\limits\sum_{u_1=b+1}^a\limits\sum_{u_2=b+1}^a\limits e_{u_1,i}t^{-s_1-1}(e_{i,u_1}t^{s_1+1},e_{u_2,j}t^{-s_1-1})e_{j,u_2}t^{s_2+1}\\
&\quad+\displaystyle\sum_{s_1,s_2\geq0}\limits\sum_{u_1=b+1}^a\limits\sum_{u_2=b+1}^a\limits e_{u_2,j}t^{-s_2-1}(e_{u_1,i}t^{-s_1-1},e_{j,u_2}t^{s_2+1})e_{i,u_1}t^{s_1+1}
\end{align*}
and
\begin{align*}
&\displaystyle\sum_{s_1,s_2\geq0}\limits\sum_{u_1=b+1}^a\limits\sum_{u_2=b+1}^a\limits e_{i,u_1}t^{-s_1}(e_{u_1,i}t^{s_1},e_{j,u_2}t^{-s_2})e_{u_2,j}t^{s_2}\\
&\quad+\displaystyle\sum_{s_1,s_2\geq0}\limits\sum_{u_1=b+1}^a\limits\sum_{u_2=b+1}^a\limits e_{j,u_2}t^{-s_2}(e_{i,u_1}t^{-s_1},e_{u_2,j}t^{s_2})e_{u_1,i}t^{s_1},
\end{align*}
where $(\ ,\ )$ is an inner product on $U(\widehat{\mathfrak{gl}}(a)^{c_a})$.
By a direct computation, these terms are equal to zero. Thus, we find that \eqref{Lemeq100} is equal to zero.
\subsection{The proof of Lemma~\ref{Lem2}}
In this subsection, we prove Lemma~\ref{Lem2}.
By the similar discussion to the one in the previous subsection, it is no problem to assume that
\begin{gather*}
[H_{i,1},e_{v,j}t^w]=[\ev_{\hbar,\ve}(H_{i,1}),e_{v,j}t^w],\\
[H_{i,1},e_{j,v}t^w]=[\ev_{\hbar,\ve}(H_{i,1}),e_{j,v}t^w].
\end{gather*}
hold in $Y^a_{\hbar,\ve}(\widehat{\mathfrak{sl}}(n))$.
We only prove the case when $i<j$. Let us compute $[A_i,F_j]$. By a direct computation, we obtain
\begin{align}
&\quad-[\hbar (e_{i+1,i+1}\otimes e_{i,i}+e_{i,i}\otimes e_{i+1,i+1}),\hbar\sum_{w\in\mathbb{Z}}\limits\sum_{v=n+1}^b e_{v,j}t^w\otimes e_{j,v}t^{-w}]\nonumber\\
&=\hbar^2\sum_{w\in\mathbb{Z}}\limits\sum_{v=n+1}^b\delta_{j,i+1}e_{v,j}t^w\otimes e_{i,i}e_{j,v}t^{-w}-\hbar^2\sum_{w\in\mathbb{Z}}\limits\sum_{v=n+1}^b\delta_{i,j}e_{v,j}t^we_{i+1,i+1}\otimes e_{j,v}t^{-w}\nonumber\\
&\quad+\hbar^2\sum_{w\in\mathbb{Z}}\limits\sum_{v=n+1}^b\delta_{j,i}e_{v,j}t^w \otimes e_{i+1,i+1}e_{j,v}t^{-w}-\hbar^2\sum_{w\in\mathbb{Z}}\limits\sum_{v=n+1}^b \delta_{j,i+1}e_{v,j}t^we_{i,i}\otimes e_{j,v}t^{-w},\label{qu1}
\end{align}
\begin{align}
&\quad[\hbar\displaystyle\sum_{s \geq 0} \limits\displaystyle\sum_{u=1}^{i}\limits (-e_{u,i}t^{-s-1}\otimes e_{i,u}t^{s+1}+e_{i,u}t^{-s}\otimes e_{u,i}t^{s}),\hbar\sum_{w\in\mathbb{Z}}\limits\sum_{v=n+1}^b e_{v,j}t^w\otimes e_{j,v}t^{-w}]\nonumber\\
&=-\hbar^2\sum_{s\geq0}\limits\sum_{w\in\mathbb{Z}}\limits\sum_{v=n+1}^b\delta(j\leq i)e_{j,i}t^{-s-1}e_{v,j}t^{w}\otimes e_{i,v}t^{s-w+1}\nonumber\\
&\quad+\hbar^2\sum_{s\geq0}\limits\sum_{w\in\mathbb{Z}}\limits\sum_{v=n+1}^b\delta(j\leq i)e_{v,i}t^{-s+w-1}\otimes e_{j,v}t^{-w}e_{i,j}t^{s+1}\nonumber\\
&\quad-\hbar^2\sum_{s\geq0}\limits\sum_{w\in\mathbb{Z}}\limits\displaystyle\sum_{u=1}^{i}\limits\sum_{v=n+1}^b\delta_{i,j}e_{v,u}t^{-s+w}\otimes e_{u,i}t^{s}e_{j,v}t^{-w}\nonumber\\
&\quad+\hbar^2\sum_{s\geq0}\limits\sum_{w\in\mathbb{Z}}\limits\displaystyle\sum_{u=1}^{i}\limits\sum_{v=n+1}^b\delta_{i,j}e_{v,j}t^{w}e_{i,u}t^{-s}\otimes e_{u,v}t^{s-w},\label{qu3}
\end{align}
\begin{align}
&\quad[\hbar\displaystyle\sum_{s \geq 0} \limits\displaystyle\sum_{u=i+1}^{n}\limits (-e_{u,i}t^{-s}\otimes e_{i,u}t^{s}+e_{i,u}t^{-s-1}\otimes e_{u,i}t^{s+1}),\hbar\sum_{w\in\mathbb{Z}}\limits\sum_{v=n+1}^b e_{v,j}t^w\otimes e_{j,v}t^{-w}]\nonumber\\
&=-\hbar^2\sum_{s\geq0}\limits\sum_{w\in\mathbb{Z}}\limits\sum_{v=n+1}^b\delta(j>i)e_{j,i}t^{-s}e_{v,j}t^{w}\otimes e_{i,v}t^{s-w}\nonumber\\
&\quad+\hbar^2\sum_{s\geq0}\limits\sum_{w\in\mathbb{Z}}\limits\sum_{v=n+1}^b\delta(j>i)e_{v,i}t^{-s+w}\otimes e_{j,v}t^{-w}e_{i,j}t^{s}\nonumber\\
&\quad-\hbar^2\sum_{s\geq0}\limits\sum_{w\in\mathbb{Z}}\limits\displaystyle\sum_{u=i+1}^{n}\limits\sum_{v=n+1}^b\delta_{i,j}e_{v,u}t^{-s+w-1}\otimes e_{u,i}t^{s+1}e_{j,v}t^{-w}\nonumber\\
&\quad+\hbar^2\sum_{s\geq0}\limits\sum_{w\in\mathbb{Z}}\limits\displaystyle\sum_{u=i+1}^{n}\limits\sum_{v=n+1}^b\delta_{i,j}e_{v,j}t^{w}e_{j,u}t^{-s-1}\otimes e_{u,v}t^{s-w+1},\label{qu4}
\end{align}
\begin{align}
&\quad-[\hbar\displaystyle\sum_{s \geq 0}\limits\displaystyle\sum_{u=1}^{i}\limits (-e_{u,i+1}t^{-s-1}\otimes e_{i+1,u}t^{s+1}+e_{i+1,u}t^{-s}\otimes e_{u,i+1}t^s),\hbar\sum_{w\in\mathbb{Z}}\limits\sum_{v=n+1}^b e_{v,j}t^w\otimes e_{j,v}t^{-w}]\nonumber\\
&=\hbar^2\sum_{s\geq0}\limits\sum_{w\in\mathbb{Z}}\limits\sum_{v=n+1}^b\delta(j\leq i)e_{j,i+1}t^{-s-1}e_{v,j}t^{w}\otimes e_{i+1,v}t^{s-w+1}\nonumber\\
&\quad-\hbar^2\sum_{s\geq0}\limits\sum_{w\in\mathbb{Z}}\limits\sum_{v=n+1}^b\delta(j\leq i)e_{v,i+1}t^{-s+w-1}\otimes e_{j,v}t^{-w}e_{i+1,j}t^{s+1}\nonumber\\
&\quad+\hbar^2\sum_{s\geq0}\limits\sum_{w\in\mathbb{Z}}\limits\displaystyle\sum_{u=1}^{i}\limits\sum_{v=n+1}^b\delta_{i+1,j}e_{v,u}t^{-s+w}\otimes e_{u,i+1}t^{s}e_{j,v}t^{-w}\nonumber\\
&\quad-\hbar^2\sum_{s\geq0}\limits\sum_{w\in\mathbb{Z}}\limits\displaystyle\sum_{u=1}^{i}\limits\sum_{v=n+1}^b\delta_{i+1,j}e_{v,j}t^{w}e_{i+1,u}t^{-s}\otimes e_{u,v}t^{s-w},\label{qu5}
\end{align}
\begin{align}
&\quad-[\hbar\displaystyle\sum_{s \geq 0}\limits\displaystyle\sum_{u=i+1}^{n} \limits (-e_{u,i+1}t^{-s}\otimes e_{i+1,u}t^{s}+e_{i+1,u}t^{-s-1}\otimes e_{u,i+1}t^{s+1},\hbar\sum_{w\in\mathbb{Z}}\limits\sum_{v=n+1}^b e_{v,j}t^w\otimes e_{j,v}t^{-w}]\nonumber\\
&=\hbar^2\sum_{s\geq0}\limits\sum_{w\in\mathbb{Z}}\limits\sum_{v=n+1}^b\delta(j>i)e_{j,i+1}t^{-s}e_{v,j}t^{w}\otimes e_{i+1,v}t^{s-w}\nonumber\\
&\quad-\hbar^2\sum_{s\geq0}\limits\sum_{w\in\mathbb{Z}}\limits\sum_{v=n+1}^b\delta(j>i)e_{v,i+1}t^{-s+w}\otimes e_{j,v}t^{-w}e_{i+1,j}t^{s}\nonumber\\
&\quad+\hbar^2\sum_{s\geq0}\limits\sum_{w\in\mathbb{Z}}\limits\displaystyle\sum_{u=i+1}^{n}\limits\sum_{v=n+1}^b\delta_{i+1,j}e_{v,u}t^{-s+w-1}\otimes e_{u,i+1}t^{s+1}e_{j,v}t^{-w}\nonumber\\
&\quad-\hbar^2\sum_{s\geq0}\limits\sum_{w\in\mathbb{Z}}\limits\displaystyle\sum_{u=i+1}^{n}\limits\sum_{v=n+1}^b\delta_{i+1,j}e_{v,j}t^{w}e_{i+1,u}t^{-s-1}\otimes e_{u,v}t^{s-w+1}.\label{qu7}
\end{align}
By the definition of $A_i$ and $F_i$, we obtain
\begin{align*}
&\quad[A_i,F_j]-[A_j,F_i]\\
&=\eqref{qu1}_{i,j}-\eqref{qu1}_{i,j+1}-\eqref{qu1}_{j,i}+\eqref{qu1}_{j,i+1}\\
&\quad+\eqref{qu3}_{i,j}-\eqref{qu3}_{i,j+1}-\eqref{qu3}_{j,i}+\eqref{qu3}_{j,i+1}\\
&\quad+\eqref{qu4}_{i,j}-\eqref{qu4}_{i,j+1}-\eqref{qu4}_{j,i}+\eqref{qu4}_{j,i+1}\\
&\quad+\eqref{qu5}_{i,j}-\eqref{qu5}_{i,j+1}-\eqref{qu5}_{j,i}+\eqref{qu5}_{j,i+1}\\
&\quad+\eqref{qu7}_{i,j}-\eqref{qu7}_{i,j+1}-\eqref{qu7}_{j,i}+\eqref{qu7}_{j,i+1}.
\end{align*}
By the assumption $i<j$, we obtain
\begin{align*}
&\quad[A_i,F_j]-[A_j,F_i]\\
&=\eqref{qu1}_{i,j,1}+\eqref{qu1}_{j,i+1,2}+\eqref{qu1}_{j,i+1,3}+\eqref{qu1}_{i,j,4}\\
&\quad-\eqref{qu3}_{j,i,1}+\eqref{qu3}_{j,i+1,1}-\eqref{qu3}_{j,i,2}+\eqref{qu3}_{j,i+1,2}+\eqref{qu3}_{j,i+1,3}+\eqref{qu3}_{j,i+1,4}\\
&\quad+\eqref{qu4}_{i,j,1}-\eqref{qu4}_{i,j+1,1}+\eqref{qu4}_{i,j,2}-\eqref{qu4}_{i,j+1,2}+\eqref{qu4}_{j,i+1,3}+\eqref{qu4}_{j,i+1,4}\\
&\quad-\eqref{qu5}_{j,i,1}+\eqref{qu5}_{j,i+1,1}-\eqref{qu5}_{j,i,2}+\eqref{qu5}_{j,i+1,2}+\eqref{qu5}_{i,j,3}+\eqref{qu5}_{i,j,4}\\
&\quad+\eqref{qu7}_{i,j,1}-\eqref{qu7}_{i,j+1,1}+\eqref{qu7}_{i,j,2}-\eqref{qu7}_{i,j+1,2}+\eqref{qu7}_{i,j,3}+\eqref{qu7}_{i,j,4}.
\end{align*}
By a direct computation, we obtain
\begin{align}
&\quad\eqref{qu5}_{i,j,4}+\eqref{qu3}_{j,i+1,4}\nonumber\\
&=-\hbar^2\sum_{s\geq0}\limits\sum_{w\in\mathbb{Z}}\limits\displaystyle\sum_{u=1}^{i}\limits\sum_{v=n+1}^b\delta_{i+1,j}e_{v,i+1}t^{w}e_{i+1,u}t^{-s}\otimes e_{u,v}t^{s-w}\nonumber\\
&\quad+\hbar^2\sum_{s\geq0}\limits\sum_{w\in\mathbb{Z}}\limits\displaystyle\sum_{u=1}^{j}\limits\sum_{v=n+1}^b\delta_{j,i+1}e_{v,j}t^{w}e_{j,u}t^{-s}\otimes e_{u,v}t^{s-w}\nonumber\\
&=\hbar^2\sum_{s\geq0}\limits\sum_{w\in\mathbb{Z}}\limits\sum_{v=n+1}^b\delta_{j,i+1}e_{v,j}t^{w}e_{j,j}t^{-s}\otimes e_{j,v}t^{s-w},\label{equation100}
\end{align}
\begin{align}
&\quad\eqref{qu7}_{i,j,4}+\eqref{qu4}_{j,i+1,4}\nonumber\\
&=-\hbar^2\sum_{s\geq0}\limits\sum_{w\in\mathbb{Z}}\limits\displaystyle\sum_{u=i+1}^{n}\limits\sum_{v=n+1}^b\delta_{i+1,j}e_{v,i+1}t^{w}e_{i+1,u}t^{-s-1}\otimes e_{u,v}t^{s-w+1}\nonumber\\
&\quad+\hbar^2\sum_{s\geq0}\limits\sum_{w\in\mathbb{Z}}\limits\displaystyle\sum_{u=j+1}^{n}\limits\sum_{v=n+1}^b\delta_{j,i+1}e_{v,j}t^{w}e_{j,u}t^{-s-1}\otimes e_{u,v}t^{s-w+1}\nonumber\\
&=-\hbar^2\sum_{s\geq0}\limits\sum_{w\in\mathbb{Z}}\limits\sum_{v=n+1}^b\delta_{i+1,j}e_{v,i+1}t^{-s-1}e_{i+1,i+1}t^{w}\otimes e_{i+1,v}t^{s-w+1}.\label{equation101}
\end{align}
By adding \eqref{equation100} and \eqref{equation101}, we have
\begin{align}
&\quad\eqref{qu5}_{i,j,4}+\eqref{qu3}_{j,i+1,4}+\eqref{qu7}_{i,j,4}+\eqref{qu4}_{j,i+1,4}\nonumber\\
&=\hbar^2\sum_{w\in\mathbb{Z}}\limits\sum_{v=n+1}^b\delta_{j,i+1}e_{v,j}t^we_{j,j}\otimes e_{j,v}t^{-w}.\label{equation111}
\end{align}
Similarly to \eqref{equation111}, we obtain
\begin{align}
&\quad\eqref{qu5}_{i,j,3}+\eqref{qu3}_{j,i+1,3}\nonumber\\
&=\hbar^2\sum_{s\geq0}\limits\sum_{w\in\mathbb{Z}}\limits\displaystyle\sum_{u=1}^{i}\limits\sum_{v=n+1}^a\delta_{i+1,j}e_{v,u}t^{-s+w}\otimes e_{u,i+1}t^{s}e_{i+1,v}t^{-w}\nonumber\\
&\quad-\hbar^2\sum_{s\geq0}\limits\sum_{w\in\mathbb{Z}}\limits\displaystyle\sum_{u=1}^{j}\limits\sum_{v=n+1}^a\delta_{j,i+1}e_{v,u}t^{-s+w}\otimes e_{u,j}t^{s}e_{j,v}t^{-w}\nonumber\\
&=-\hbar^2\sum_{s\geq0}\limits\sum_{w\in\mathbb{Z}}\limits\sum_{v=n+1}^a\delta_{j,i+1}e_{v,j}t^{-s+w}\otimes e_{j,j}t^{s}e_{j,v}t^{-w},\label{equation102}
\end{align}
\begin{align}
&\quad\eqref{qu7}_{i,j,3}+\eqref{qu4}_{j,i+1,3}\nonumber\\
&=\hbar^2\sum_{s\geq0}\limits\sum_{w\in\mathbb{Z}}\limits\displaystyle\sum_{u=i+1}^{n}\limits\sum_{v=n+1}^a\delta_{i+1,j}e_{v,u}t^{-s+w-1}\otimes e_{u,i+1}t^{s+1}e_{i+1,v}t^{-w}\nonumber\\
&\quad-\hbar^2\sum_{s\geq0}\limits\sum_{w\in\mathbb{Z}}\limits\displaystyle\sum_{u=j+1}^{n}\limits\sum_{v=n+1}^a\delta_{j,i+1}e_{v,u}t^{-s+w-1}\otimes e_{u,j}t^{s+1}e_{j,v}t^{-w}\nonumber\\
&=\hbar^2\sum_{s\geq0}\limits\sum_{w\in\mathbb{Z}}\limits\displaystyle\sum_{v=n+1}^a\delta_{i+1,j}e_{v,i+1}t^{-s+w-1}\otimes e_{i+1,i+1}t^{s+1}e_{i+1,v}t^{-w}.\label{equation103}
\end{align}
By adding \eqref{equation102} and \eqref{equation103}, we have
\begin{align}
&\quad\eqref{qu5}_{i,j,3}+\eqref{qu3}_{j,i+1,3}+\eqref{qu7}_{i,j,3}+\eqref{qu4}_{j,i+1,3}\nonumber\\
&=-\hbar^2\sum_{w\in\mathbb{Z}}\limits\sum_{v=n+1}^a\delta_{j,i+1}e_{v,j}t^{w}\otimes e_{j,j}e_{j,v}t^{-w}.\label{equation112}
\end{align}
Then, we obtain
\begin{align*}
&\quad[A_i,F_j]-[A_j,F_i]\\
&=\eqref{qu1}_{i,j,1}+\eqref{qu1}_{i,j,4}+\eqref{qu1}_{j,i+1,2}+\eqref{qu1}_{j,i+1,3}\\
&\quad-\eqref{qu3}_{j,i,1}-\eqref{qu3}_{j,i,2}+\eqref{qu3}_{j,i+1,1}+\eqref{qu3}_{j,i+1,2}\\
&\quad+\eqref{qu4}_{i,j,1}+\eqref{qu4}_{i,j,2}-\eqref{qu4}_{i,j+1,1}-\eqref{qu4}_{i,j+1,2}\\
&\quad-\eqref{qu5}_{j,i,1}-\eqref{qu5}_{j,i,2}+\eqref{qu5}_{j,i+1,1}+\eqref{qu5}_{j,i+1,2}\\
&\quad+\eqref{qu7}_{i,j,1}+\eqref{qu7}_{i,j,2}-\eqref{qu7}_{i,j+1,1}-\eqref{qu7}_{i,j+1,2}+\eqref{equation111}+\eqref{equation112}.
\end{align*}

Next, let us compute $[(H_{i,1}+B_i)\otimes1,F_j]-[(H_{j,1}+B_j)\otimes1,F_i]$. By the defining relation,
By a direct computation, we obtain
\begin{align}
&\quad[(H_{i,1}+B_i)\otimes1,\hbar\sum_{w\in\mathbb{Z}}\limits\sum_{v=n+1}^b e_{v,j}t^w\otimes e_{j,v}t^{-w}]\nonumber\\
&=\dfrac{i}{2}\sum_{w\in\mathbb{Z}}\limits\sum_{v=n+1}^b\hbar^2\delta_{i,j}e_{v,j}t^w\otimes e_{j,v}t^{-w}-\dfrac{i}{2}\sum_{w\in\mathbb{Z}}\limits\sum_{v=n+1}^b\hbar^2\delta_{i+1,j}e_{v,j}t^w\otimes e_{j,v}t^{-w}\nonumber\\
&\quad+\sum_{w\in\mathbb{Z}}\limits\sum_{v=n+1}^b\hbar^2\delta_{i,j}e_{v,j}t^we_{i+1,i+1}\otimes e_{j,v}t^{-w}\nonumber\\
&\quad+\sum_{w\in\mathbb{Z}}\limits\sum_{v=n+1}^b\hbar^2\delta_{i+1,j}e_{i,i}e_{v,j}t^w\otimes e_{j,v}t^{-w}\nonumber\\
&\quad-\hbar^2\sum_{w\in\mathbb{Z}}\limits\sum_{v=n+1}^b\delta(j\leq i)\displaystyle\sum_{s \geq 0}\limits e_{i,j}t^{-s}e_{v,i}t^{s+w}\otimes e_{j,v}t^{-w}\nonumber\\
&\quad-\hbar^2\sum_{w\in\mathbb{Z}}\limits\sum_{v=n+1}^b\displaystyle\sum_{s \geq 0}\limits\sum_{u=1}^i\limits\delta_{i,j}e_{v,u}t^{w-s}e_{u,i}t^s\otimes e_{j,v}t^{-w}\nonumber\\
&\quad-\hbar^2\sum_{w\in\mathbb{Z}}\limits\sum_{v=n+1}^b\displaystyle\sum_{s \geq 0} \limits\delta(j>i)e_{i,j}t^{-s-1}e_{v,i}t^{s+w+1}\otimes e_{j,v}t^{-w} \nonumber\\
&\quad-\hbar^2\sum_{w\in\mathbb{Z}}\limits\sum_{v=n+1}^b\displaystyle\sum_{s \geq 0} \limits\displaystyle\sum_{u=i+1}^{n}\limits\delta_{i,j}e_{v,u}t^{w-s-1}e_{u,i}t^{s+1}\otimes e_{j,v}t^{-w}\nonumber\\
&\quad+\hbar^2\sum_{w\in\mathbb{Z}}\limits\sum_{v=n+1}^b\displaystyle\sum_{s \geq 0}\limits\displaystyle\sum_{u=1}^{i}\limits \delta(j\leq i)e_{i+1,j}t^{-s}e_{v,i+1}t^{s+w}\otimes e_{j,v}t^{-w}\nonumber\\
&\quad+\hbar^2\sum_{w\in\mathbb{Z}}\limits\sum_{v=n+1}^b\displaystyle\sum_{s \geq 0}\limits\displaystyle\sum_{u=1}^{i}\limits \delta_{i+1,j}e_{v,u}t^{w-s}e_{u,i+1}t^s\otimes e_{j,v}t^{-w}\nonumber\\
&\quad+\hbar^2\sum_{w\in\mathbb{Z}}\limits\sum_{v=n+1}^b\displaystyle\sum_{s \geq 0}\limits\delta(j>i)e_{i+1,j}t^{-s-1}e_{v,i+1}t^{s+w+1}\otimes e_{j,v}t^{-w}\nonumber\\
&\quad+\hbar^2\sum_{w\in\mathbb{Z}}\limits\sum_{v=n+1}^b\displaystyle\sum_{s \geq 0}\limits\displaystyle\sum_{u=i+1}^{n} \limits \delta_{i+1,j}e_{v,u}t^{w-s-1}e_{u,i+1}t^{s+1}\otimes e_{j,v}t^{-w}\nonumber\\
&\quad-\delta_{i,j}\hbar^2\displaystyle\sum_{s\geq0}\limits\displaystyle\sum_{u=b+1}^a\limits\sum_{w\in\mathbb{Z}}\limits\sum_{v=n+1}^b e_{u,i}t^{-s-1}e_{v,u}t^{s+w+1}\otimes e_{j,v}t^{-w}\nonumber\\
&\quad-\delta_{i,j}\hbar^2\displaystyle\sum_{s\geq0}\limits\displaystyle\sum_{u=b+1}^a\limits\sum_{w\in\mathbb{Z}}\limits\sum_{v=n+1}^b e_{v,u}t^{w-s}e_{u,i}t^{s}\otimes e_{j,v}t^{-w}\nonumber\\
&\quad+\delta_{i+1,j}\hbar^2\displaystyle\sum_{s\geq0}\limits\displaystyle\sum_{u=b+1}^a\limits\sum_{w\in\mathbb{Z}}\limits\sum_{v=n+1}^b e_{u,i+1}t^{-s-1}e_{v,u}t^{s+w+1}\otimes e_{j,v}t^{-w}\nonumber\\
&\quad-\delta_{i+1,j}\hbar^2\displaystyle\sum_{s\geq0}\limits\displaystyle\sum_{u=b+1}^a\limits\sum_{w\in\mathbb{Z}}\limits\sum_{v=n+1}^b e_{v,u}t^{w-s}e_{u,i+1}t^{s}\otimes e_{j,v}t^{-w}.\label{equation1}
\end{align}
By the assumption $i<j$, we have
\begin{align*}
&\quad[(H_{i,1}+B_i)\otimes 1,F_j]-[(H_{j,1}+B_i)\otimes 1,F_i]\\
&=\eqref{equation1}_{j,i+1,1}+\eqref{equation1}_{i,j,2}+\eqref{equation1}_{j,i+1,3}+\eqref{equation1}_{j,i+1,4}\\
&\quad-\eqref{equation1}_{j,i,5}+\eqref{equation1}_{j,i+1,5}+\eqref{equation1}_{j,i+1,6}\\
&\quad+\eqref{equation1}_{i,j,7}-\eqref{equation1}_{i,j+1,7}+\eqref{equation1}_{j,i+1,8}-\eqref{equation1}_{j,i,9}+\eqref{equation1}_{j,i+1,9}+\eqref{equation1}_{i,j,10}\\
&\quad+\eqref{equation1}_{i,j,11}-\eqref{equation1}_{i,j+1,11}+\eqref{equation1}_{i,j,12}\\
&\quad+\eqref{equation1}_{j,i+1,13}+\eqref{equation1}_{j,i+1,14}+\eqref{equation1}_{i,j,15}+\eqref{equation1}_{i,j,16}.
\end{align*}
By a direct computation, we obtain
\begin{align}
&\quad\eqref{equation1}_{j,i+1,1}+\eqref{equation1}_{i,j,2}\nonumber\\
&=\dfrac{j}{2}\sum_{w\in\mathbb{Z}}\limits\sum_{v=n+1}^b\hbar^2\delta_{j,i+1}e_{v,i+1}t^w\otimes e_{i+1,v}t^{-w}-\dfrac{i}{2}\sum_{w\in\mathbb{Z}}\limits\sum_{v=n+1}^b\hbar^2\delta_{i+1,j}e_{v,j}t^w\otimes e_{j,v}t^{-w}\nonumber\\
&=\dfrac{1}{2}\sum_{w\in\mathbb{Z}}\limits\sum_{v=n+1}^b\hbar^2\delta_{j,i+1}e_{v,i+1}t^w\otimes e_{i+1,v}t^{-w},\label{equation201}
\end{align}
By a direct computation, we obtain
\begin{gather}
\eqref{equation1}_{j,i+1,13}+\eqref{equation1}_{i,j,15}=0,\\
\eqref{equation1}_{j,i+1,14}+\eqref{equation1}_{i,j,16}=0.
\end{gather}
By using
\begin{align*}
&\quad\eqref{equation1}_{j,i+1,6}+\eqref{equation1}_{i,j,10}\nonumber\\
&=-\hbar^2\sum_{w\in\mathbb{Z}}\limits\sum_{v=n+1}^b\displaystyle\sum_{s \geq 0}\limits\sum_{u=1}^{j}\limits\delta_{j,i+1}e_{v,u}t^{w-s}e_{u,j}t^s\otimes e_{i+1,v}t^{-w}\nonumber\\
&\quad+\hbar^2\sum_{w\in\mathbb{Z}}\limits\sum_{v=n+1}^b\displaystyle\sum_{s \geq 0}\limits\displaystyle\sum_{u=1}^{i}\limits \delta_{i+1,j}e_{v,u}t^{w-s}e_{u,i+1}t^s\otimes e_{j,v}t^{-w}\nonumber\\
&=-\hbar^2\sum_{w\in\mathbb{Z}}\limits\sum_{v=n+1}^b\displaystyle\sum_{s \geq 0}\limits\delta_{j,i+1}e_{v,j}t^{w-s}e_{j,j}t^s\otimes e_{i+1,v}t^{-w}
\end{align*}
and
\begin{align*}
&\quad\eqref{equation1}_{j,i+1,8}+\eqref{equation1}_{i,j,12}\nonumber\\
&=-\hbar^2\sum_{w\in\mathbb{Z}}\limits\sum_{v=n+1}^b\displaystyle\sum_{s \geq 0} \limits\displaystyle\sum_{u=j+1}^{n}\limits\delta_{i+1,j}e_{i+1,v}t^{-w}\otimes e_{v,u}t^{w-s-1}e_{u,j}t^{s+1}\nonumber\\
&\quad+\hbar^2\sum_{w\in\mathbb{Z}}\limits\sum_{v=n+1}^b\displaystyle\sum_{s \geq 0}\limits\displaystyle\sum_{u=i+1}^{n} \limits \delta_{i+1,j}e_{v,u}t^{w-s-1}e_{u,i+1}t^{s+1}\otimes e_{j,v}t^{-w}\nonumber\\
&=\hbar^2\sum_{w\in\mathbb{Z}}\limits\sum_{v=n+1}^b\displaystyle\sum_{s \geq 0}\limits\delta_{i+1,j}e_{v,i+1}t^{w-s-1}e_{i+1,i+1}t^{s+1}\otimes e_{j,v}t^{-w},
\end{align*}
we find that
\begin{align}
&\quad\eqref{equation1}_{j,i+1,6}+\eqref{equation1}_{i,j,10}+\eqref{equation1}_{j,i+1,8}+\eqref{equation1}_{i,j,12}\nonumber\\
&=-\hbar^2\sum_{w\in\mathbb{Z}}\limits\sum_{v=n+1}^b\delta_{j,i+1}e_{v,j}t^{w}e_{j,j}\otimes e_{i+1,v}t^{-w}.\label{equation202}
\end{align}
By a direct computation, we obtain
\begin{align}
&\quad\eqref{equation111}+\eqref{equation202}=0,\\
&\quad\eqref{qu1}_{j,i+1,2}+\eqref{equation1}_{j,i+1,3}=0,\\
&\quad\eqref{qu1}_{i,j,4}+\eqref{equation1}_{i,j,4}=0,\\
&\quad-\eqref{qu3}_{j,i,1}+\eqref{equation1}_{i,j,7}=0,\\
&\quad\eqref{qu4}_{i,j,1}-\eqref{equation1}_{j,i,5}=0,\\
&\quad-\eqref{qu5}_{j,i,1}-\eqref{equation1}_{i,j+1,7}=0,\\
&\quad\eqref{qu7}_{i,j,1}+\eqref{equation1}_{j,i+1,5}=0,\\
&\quad\eqref{qu3}_{j,i+1,1}+\eqref{equation1}_{i,j,11}=0,\\
&\quad-\eqref{qu4}_{i,j+1,1}-\eqref{equation1}_{j,i,9}=0,\\
&\quad\eqref{qu5}_{j,i+1,1}-\eqref{equation1}_{i,j+1,11}=0,\\
&\quad-\eqref{qu7}_{i,j+1,1}+\eqref{equation1}_{j,i+1,9}=0.\label{saigo}
\end{align}
Then, we find that
\begin{align}
&\quad[(H_{i,1}+B_i)\otimes 1,F_j]-[(H_{j,1}+B_j)\otimes 1,F_i]\nonumber\\
&=\eqref{qu1}_{i,j,1}+\eqref{qu1}_{i,j,4}-\eqref{qu3}_{j,i,2}+\eqref{qu3}_{j,i+1,2}\nonumber\\
&\quad+\eqref{qu4}_{i,j,2}-\eqref{qu4}_{i,j+1,2}-\eqref{qu5}_{j,i,2}+\eqref{qu5}_{j,i+1,2}\nonumber\\
&\quad+\eqref{qu7}_{i,j,2}-\eqref{qu7}_{i,j+1,2}+\eqref{equation201}.\label{saisyo}
\end{align}

Next, let us compute $[1\otimes H_{i,1},F_j]-[1\otimes H_{i,1},F_i]$.
By a direct computation, we obtain
\begin{align}
&\quad[1\otimes H_{i,1},\hbar\sum_{w\in\mathbb{Z}}\limits\sum_{v=n+1}^b\limits e_{v,j}t^{-w}\otimes e_{j,v}t^w]\nonumber\\
&=-\dfrac{i}{2}\hbar^2\sum_{w\in\mathbb{Z}}\limits\sum_{v=n+1}^b\limits\delta_{i,j}e_{v,j}t^{-w}\otimes e_{j,v}t^w+\dfrac{i}{2}\hbar^2\sum_{w\in\mathbb{Z}}\limits\sum_{v=n+1}^b\limits\delta_{i+1,j}e_{v,j}t^{-w}\otimes e_{j,v}t^w\nonumber\\
&\quad-\hbar^2 \sum_{w\in\mathbb{Z}}\limits\sum_{v=n+1}^b\limits\delta_{i,j}e_{v,j}t^{-w}\otimes e_{j,v}t^we_{i+1,i+1}-\hbar^2 \sum_{w\in\mathbb{Z}}\limits\sum_{v=n+1}^b\limits \delta_{i+1,j}e_{v,j}t^{-w}\otimes e_{i,i}e_{j,v}t^w\nonumber\\
&\quad+\hbar^2\displaystyle\sum_{s \geq 0}  \limits\displaystyle\sum_{u=1}^{i}\limits\sum_{w\in\mathbb{Z}}\limits\sum_{v=n+1}^b\limits \delta_{i,j}e_{v,j}t^{-w}\otimes e_{i,u}t^{-s}e_{u,v}t^{s+w}\nonumber\\
&\quad+\hbar^2\displaystyle\sum_{s \geq 0}  \limits\sum_{w\in\mathbb{Z}}\limits\sum_{v=n+1}^b\limits\delta(j\leq i)e_{v,j}t^{-w}\otimes e_{i,v}t^{w-s}e_{j,i}t^s\nonumber\\
&\quad+\hbar^2\displaystyle\sum_{s \geq 0} \limits\displaystyle\sum_{u=i+1}^{a}\limits\sum_{w\in\mathbb{Z}}\limits\sum_{v=n+1}^b\limits \delta_{i,j}e_{v,j}t^{-w}\otimes e_{i,u}t^{-s-1}e_{u,v}t^{s+w+1}\nonumber\\
&\quad+\hbar^2\displaystyle\sum_{s \geq 0} \limits\sum_{w\in\mathbb{Z}}\limits\sum_{v=n+1}^b\limits\delta(j>i)e_{v,j}t^{-w}\otimes e_{i,v}t^{w-s-1}e_{j,i}t^{s+1}\nonumber\\
&\quad-\hbar^2\displaystyle\sum_{s \geq 0}\limits\displaystyle\sum_{u=1}^{i}\limits \sum_{w\in\mathbb{Z}}\limits\sum_{v=n+1}^b\limits \delta_{i+1,j}e_{v,j}t^{-w}\otimes e_{i+1,u}t^{-s}e_{u,v}t^{s+w}\nonumber\\
&\quad-\hbar^2\displaystyle\sum_{s \geq 0}\limits\sum_{w\in\mathbb{Z}}\limits\sum_{v=n+1}^b\limits \delta(j\leq i)e_{v,j}t^{-w}\otimes e_{i+1,v}t^{w-s}e_{j,i+1}t^s\nonumber\\
&\quad-\hbar^2\displaystyle\sum_{s \geq 0}\limits\displaystyle\sum_{u=i+1}^{a} \limits\sum_{w\in\mathbb{Z}}\limits\sum_{v=n+1}^b\limits \delta_{i+1,j}e_{v,j}t^{-w}\otimes e_{i+1,u}t^{-s-1}e_{u,v}t^{s+w+1}\nonumber\\
&\quad-\hbar^2\displaystyle\sum_{s \geq 0}\limits\sum_{w\in\mathbb{Z}}\limits\sum_{v=n+1}^b\limits\delta(j>i)e_{v,j}t^{-w}\otimes e_{i+1,v}t^{w-s-1}e_{j,i+1}t^{s+1}.\label{equation2}
\end{align}
By the assumption $i<j$, we obtain
\begin{align*}
&\quad[1\otimes H_{i,1},F_j]-[1\otimes H_{j,1},F_i]\\
&=\eqref{equation2}_{j,i+1,1}+\eqref{equation2}_{i,j,2}+\eqref{equation2}_{j,i+1,3}+\eqref{equation2}_{i,j,4}+\eqref{equation2}_{j,i+1,5}\\
&\quad-\eqref{equation2}_{j,i,6}+\eqref{equation2}_{j,i+1,6}+\eqref{equation2}_{j,i+1,7}+\eqref{equation2}_{i,j,8}-\eqref{equation2}_{i,j+1,8}\\
&\quad+\eqref{equation2}_{i,j,9}-\eqref{equation2}_{j,i,10}+\eqref{equation2}_{j,i+1,10}+\eqref{equation2}_{i,j,11}+\eqref{equation2}_{i,j,12}-\eqref{equation2}_{i,j+1,12}.
\end{align*}
By a direct computation, we obtain
\begin{align}
&\quad\eqref{equation2}_{j,i+1,1}+\eqref{equation2}_{i,j,2}\nonumber\\
&=-\dfrac{j}{2}\hbar^2\sum_{w\in\mathbb{Z}}\limits\sum_{v=n+1}^b\limits\delta_{j,i+1}e_{i+1,v}t^w\otimes e_{v,i+1}t^{-w}+\dfrac{i}{2}\hbar\sum_{w\in\mathbb{Z}}\limits\sum_{v=n+1}^b\limits\delta_{i+1,j}e_{j,v}t^w\otimes e_{v,j}t^{-w}\nonumber\\
&=-\dfrac{1}{2}\hbar^2\sum_{w\in\mathbb{Z}}\limits\sum_{v=n+1}^b\limits\delta_{j,i+1}e_{i+1,v}t^w\otimes e_{v,i+1}t^{-w}.\label{equation505}
\end{align}
By a direct computation, we obtain
\begin{align}
\eqref{equation505}+\eqref{equation201}=0.
\end{align}
By using
\begin{align}
&\quad\eqref{equation2}_{j,i+1,5}+\eqref{equation2}_{i,j,9}\nonumber\\
&=\hbar^2\displaystyle\sum_{s \geq 0}  \limits\displaystyle\sum_{u=1}^{j}\limits\sum_{w\in\mathbb{Z}}\limits\sum_{v=n+1}^b\limits \delta_{i+1,j}e_{v,i+1}t^{-w}\otimes e_{j,u}t^{-s}e_{u,v}t^{s+w}\nonumber\\
&\quad-\hbar^2\displaystyle\sum_{s \geq 0}\limits\displaystyle\sum_{u=1}^{i}\limits \sum_{w\in\mathbb{Z}}\limits\sum_{v=n+1}^b\limits \delta_{i+1,j}e_{v,j}t^{-w}\otimes e_{i+1,u}t^{-s}e_{u,v}t^{s+w}\nonumber\\
&=\hbar^2\displaystyle\sum_{s \geq 0}  \limits\displaystyle\sum_{w\in\mathbb{Z}}\limits\sum_{v=n+1}^b\limits \delta_{i+1,j}e_{v,i+1}t^{-w}\otimes e_{j,j}t^{-s}e_{j,v}t^{s+w}
\end{align}
and
\begin{align}
&\quad\eqref{equation2}_{j,i+1,7}+\eqref{equation2}_{i,j,11}\nonumber\\
&=\hbar^2\displaystyle\sum_{s \geq 0} \limits\displaystyle\sum_{u=j+1}^{n}\limits\sum_{w\in\mathbb{Z}}\limits\sum_{v=n+1}^b\limits \delta_{i+1,j}e_{v,j}t^{-w}\otimes e_{i+1,u}t^{-s-1}e_{u,v}t^{s+w+1}\nonumber\\
&\quad-\hbar^2\displaystyle\sum_{s \geq 0}\limits\displaystyle\sum_{u=i+1}^{n} \limits\sum_{w\in\mathbb{Z}}\limits\sum_{v=n+1}^b\limits \delta_{i+1,j}e_{v,j}t^{-w}\otimes e_{i+1,u}t^{-s-1}e_{u,v}t^{s+w+1}\nonumber\\
&=-\hbar^2\displaystyle\sum_{s \geq 0}\limits\displaystyle\sum_{w\in\mathbb{Z}}\limits\sum_{v=n+1}^b\limits \delta_{i+1,j} e_{v,j}t^{-w}\otimes e_{i+1,i+1}t^{-s-1}e_{i+1,v}t^{s+w+1}.
\end{align}
we obtain
\begin{align}
&\quad\eqref{equation2}_{j,i+1,5}+\eqref{equation2}_{i,j,9}+\eqref{equation2}_{j,i+1,7}+\eqref{equation2}_{i,j,11}\nonumber\\
&=\hbar^2\displaystyle\sum_{w\in\mathbb{Z}}\limits\sum_{v=n+1}^b\limits \delta_{i+1,j}e_{v,i+1}t^{-w}\otimes e_{j,j}e_{j,v}t^{w}.\label{equation301}
\end{align}

By a direct computation, we obtain
\begin{align}
&\quad\eqref{equation301}+\eqref{equation112}=0,\\
&\quad\eqref{equation2}_{j,i+1,3}+\eqref{qu1}_{j,i+1,3}=0,\\
&\quad\eqref{equation2}_{i,j,4}+\eqref{qu1}_{i,j,1}=0,\\
&\quad-\eqref{qu3}_{j,i,2}+\eqref{equation2}_{i,j,8}=0,\\
&\quad\eqref{qu3}_{j,i+1,2}+\eqref{equation2}_{i,j,12}=0,\\
&\quad\eqref{qu7}_{i,j,2}+\eqref{equation2}_{j,i+1,6}=0,\\
&\quad-\eqref{qu7}_{i,j+1,2}+\eqref{equation2}_{j,i+1,10}=0,\\
&\quad\eqref{qu4}_{i,j,2}-\eqref{equation2}_{j,i,6}=0,\\
&\quad-\eqref{qu4}_{i,j+1,2}-\eqref{equation2}_{j,i,10}=0,\\
&\quad\eqref{qu5}_{j,i,2}+\eqref{equation2}_{i,j+1,8}=0,\\
&\quad\eqref{qu5}_{j,i+1,2}+\eqref{equation2}_{i,j+1,12}=0.
\end{align}
This completes the proof of Lemma~\ref{Lem2}.
\section*{Acknowledgement}

The author wishes to express his gratitude to Daniele Valeri. This article is inspired by his lecture at "Quantum symmetries: Tensor categories, Topological quantum field theories, Vertex algebras" held at the University of Montreal. The author is also grateful to
Thomas Creutzig for proposing this problem. The author expresses his sincere thanks to Nicolas Guay and Shigenori Nakatsuka for the useful advice and discussions.
\bibliographystyle{plain}
\bibliography{syuu}
\end{document}